\documentclass[12pt,paper,reqno]{amsart}

\numberwithin{equation}{section}
\usepackage{verbatim} 

\usepackage{amsfonts, amssymb, amsthm, amsmath, bbm, wasysym, MnSymbol,enumerate,euscript}
\usepackage[shortlabels]{enumitem}
\usepackage[all,cmtip]{xy}
\usepackage{youngtab}
\usepackage[dvipsnames]{xcolor}
\usepackage{mathtools}
\usepackage{multicol}

\theoremstyle{definition}
\newtheorem{theorem}{Theorem}
\newtheorem{lemma}[theorem]{Lemma}
\newtheorem{proposition}[theorem]{Proposition}
\newtheorem{corollary}[theorem]{Corollary}
\newtheorem{definition}[theorem]{Definition}
\newtheorem{remark}[theorem]{Remark}

\newtheorem{example}[theorem]{Example}

\newcommand{\obj}{\filledstar}

\newcommand{\op}{\operatorname}
\DeclareMathOperator{\gr}{gr}
\DeclareMathOperator{\End}{\op{End}}
\DeclareMathOperator{\Conn}{\op{Conn}}
\DeclareMathOperator{\Hom}{{Hom}}
\DeclareMathOperator{\str}{str}
\DeclareMathOperator{\tr}{tr}

\usepackage{tikz}

\usepackage{enumitem}

\usepackage{comment}

\usepgflibrary{shapes.geometric}

\tikzstyle{V}=[draw, fill =black, circle, inner sep=0pt, minimum size=1.5pt]
\tikzstyle{C}=[draw, fill =white, circle, inner sep=0pt, minimum size=1.5pt]
\tikzstyle{over}=[draw=white,double=black,line width=2pt, double distance=.5pt]

 \def\CC{\mathbb{C}}
 \def\fgl{\mathfrak{gl}} 
 \def\fp{\mathfrak{p}}
 \def\cX{\mathcal{X}}
 \def\<{\langle}
  \def\>{\rangle}
  \def\half{\frac{1}{2}}
  \def\mC{\mathbb{C}}

\newcommand{\nc}{\newcommand}
\nc{\sVW}{\mathit{s}\hspace{-0.7mm}\bigdoublevee}
\nc{\VW}{\bigdoublevee}
\nc{\VWd}{{\bigdoublevee}_d}
\nc{\sBr}{\mathit{s}\mathcal{B}\mathit{r}}
\nc{\GVW}{\mathit{gs}\hspace{-0.7mm}\bigdoublevee}
\nc{\pn}{\mathfrak{p}(n)-\text{mod}}
\nc{\C}{\mathbb{C}}
\nc{\Z}{\mathbb{Z}}
\nc{\s}{\sigma}
\nc{\NN}{\mathbb{N}}
\nc{\Bra}{B^{(r)}_a}
\nc{\Cra}{C^{(r)}_a}
\nc{\FVect}{\C\mathrm{-fmod}}
\nc{\GVect}{\C\mathrm{-gmod}}

\nc{\icom}{\color{orange} Iva: }

\newcommand\TikZ[1]{\begin{matrix}\begin{tikzpicture}#1\end{tikzpicture}\end{matrix}}

\allowdisplaybreaks

\begin{document}

\title{The affine VW supercategory}
 
\author[]{M. Balagovi\'c 
\and Z. Daugherty 
\and I. Entova-Aizenbud 
\and I. Halacheva 
\and J. Hennig 
\and M. S. Im 
\and G. Letzter 
\and E. Norton 
\and V. Serganova 
\and C. Stroppel}

\address{M. B.: School of Mathematics, Statistics and Physics, Newcastle University, Newcastle upon Tyne, NE1 7RU UK, \it{martina.balagovic@newcastle.ac.uk}}
\address{Z. D.: Department of Mathematics, City College of New York, New York, NY 10031 USA,  \it{zdaugherty@gmail.com}}
\address{I. E.: Department of Mathematics, Ben-Gurion University, Beer-Sheva, Israel, \it{inna.entova@gmail.com}}
\address{I. H.: Department of Mathematics and Statistics, Lancaster University, Lancaster, LA1 4YF UK, \it{ihalacheva@gmail.com}}
\address{J. H.: Center for Communications Research, San Diego, CA 92121 USA, \it{jhennigmath@gmail.com}}
\address{M. S. I.: Department of Mathematical Sciences, United States Military Academy, West Point, NY 10996 USA, \it{meeseongim@gmail.com}}
\address{G. L.: Department of Defense, Ft. George G. Meade, MD 20755 USA, \it{gletzter@verizon.net}}
\address{E. N.: Max Planck Institute for Mathematics, Vivatsgasse 7, 53111 Bonn, Germany, \it{enorton@mpim-bonn.mpg.de}}
\address{V. S.: Department of Mathematics, University of California at Berkeley, Berkeley, CA 94720 USA, \it{serganov@math.berkeley.edu}}
\address{C. S.: Hausdorff Center, University of Bonn, Endenicher Allee 60,  53115 Bonn, Germany, \it{stroppel@math.uni-bonn.de}}

\begin{abstract}	
	We define the affine VW supercategory $\sVW$, which arises from studying the action of the periplectic Lie superalgebra $\fp(n)$ on the tensor product $M\otimes V^{\otimes a}$ of an arbitrary representation $M$ with several copies of the vector representation $V$
	of $\fp(n)$. It plays a role analogous to that of the degenerate affine Hecke algebras in the context of representations of the general linear group; the main obstacle was the lack of a quadratic Casimir element in $\fp(n)\otimes \fp(n)$. When $M$ is the trivial representation, the action factors through the Brauer supercategory $\sBr$. Our main result is an explicit basis theorem for the morphism spaces of $\sVW$ and, as a consequence, of $\sBr$. The proof utilises the close connection with the representation theory of $\fp(n)$. As an application we explicitly describe the centre of all endomorphism algebras, and show that it behaves well under the passage to the associated graded and under deformation.
\end{abstract}

\maketitle 
\vspace{-0.6cm}
\tableofcontents
\vspace{-0.6cm}

\addtocontents{toc}{\protect\setcounter{tocdepth}{0}}
\section*{Introduction}

\subsection*{Classical and higher Schur-Weyl duality}
Classical and higher Schur-Weyl dualities are important tools in representation theory. Working over the fixed ground field $\mathbb{C}$, the classical {\it Schur-Weyl duality} for the general linear Lie algebra $\mathfrak{gl}_n$ refers to the double centralizer theorem applied to the commuting actions of $\mathfrak{gl}_n$ and the symmetric group $S_a$ 
\begin{align}
\mathfrak{gl}_n\;\curvearrowright V^{\otimes a}\;\curvearrowleft S_a \label{SWtriv}
\end{align}
on the tensor product of $a$ copies of the vector representation $V$. By {\it (higher) Schur-Weyl duality} (see \cite{AS}, \cite{BK}) we mean the existence of commuting actions 
\begin{align}
\mathfrak{gl}_n\;\curvearrowright M\otimes V^{\otimes a}\;\curvearrowleft H_a \label{SW}
\end{align}
of $\mathfrak{gl}_n$ and the degenerate affine Hecke algebra $H_a$ on the tensor product of an arbitrary $\mathfrak{gl}_n$-representation $M$ with $V^{\otimes a}$. The degenerate affine Hecke algebra $H_a$, introduced by Drinfeld \cite{Drinfeld} and Lusztig \cite{Lusztig}, contains the group algebra $\mathbb{C}[S_a]$ and the polynomial algebra $\mathbb{C}[y_1,\ldots, y_a]$ as subalgebras, and is isomorphic as vector space to $\mathbb{C}[S_a]\otimes \mathbb{C}[y_1,\ldots, y_a]$. In particular it has a basis 
\begin{eqnarray*}
\mathcal{B}&=&\{wy_1^{k_1}\cdots y_a^{k_a}\mid w\in S_a, k_i\in \mathbb{N}_0\}.
\end{eqnarray*}
The action of the symmetric group on $M\otimes V^{\otimes a}$ is 
 given by permuting the tensor factors of $V^{\otimes a}$. To get the action of the polynomial generators $y_i$, one additionally considers the Casimir element 
\begin{eqnarray}
\label{Casimir}
\Omega^{\mathfrak{gl}_n}&=&\sum_{1\leq i,j\leq n}E_{ij}\otimes E_{ji}\in \mathfrak{gl}_n\otimes \mathfrak{gl}_n,
\end{eqnarray}
labels the tensor factors of $M\otimes V^{\otimes a}$ as $0,1,\ldots, a$, and then sets
\begin{eqnarray}
\label{yi-action-Hecke}
y_i=\sum_{j=0}^{i-1}\Omega_{ji}^{\mathfrak{gl}_n},
\end{eqnarray}
with $\Omega_{ji}$ denoting the action of $\Omega$ on the $j$-th and $i$-th tensor factors of $M\otimes V^{\otimes a}$. 
These operators satisfy $y_{i+1}=s_iy_is_i+ s_i$ for $s_i=(i,i+1)\in S_a$, and  define an action of $H_a$. When $M$ is the trivial representation, this action factors through the quotient $H_a\to \mathbb{C}[S_a]$, and (\ref{SW}) reduces to (\ref{SWtriv}). The quotient map $H_a\to \mathbb{C}[S_a]$ sends $y_1,\ldots, y_a$ to the Jucys-Murphy elements of $\mathbb{C}[S_a]$. 

The existence of \eqref{SWtriv} and \eqref{SW} allows one to pass knowledge about the representation theory between the two sides of the duality. It is also crucial for the construction and definition of $2$-Kac Moody representations in the sense of Rouquier, \cite{Rouquier}.

\subsection*{Commuting actions for the periplectic Lie superalgebras $\mathfrak{p}(n)$}

We aim to establish a duality analogous to \eqref{SW} in a situation where $\mathfrak{gl}_n$ is replaced by the {\it periplectic Lie superalgebra} $\mathfrak{p}(n)$. The family $\mathfrak{p}(n)$, $n\ge 2$, is the first family of so-called ``strange'' Lie superalgebras in the classification of reductive Lie superalgebras \cite{Kac}. The hope is to use a duality like \eqref{SW} as a tool in understanding the representation theory of $\mathfrak{p}(n)$.

The superalgebra $\mathfrak{p}(n)$ is defined as the subalgebra of the general linear superalgebra $\mathfrak{gl}(n|n)$, consisting of all elements preserving a certain bilinear form $\beta$ on the vector representation $V$ of $\mathfrak{gl}(n|n)$ (see Section \ref{sec:periplectic} for the definition). The duality analogous to \eqref{SWtriv} has been established in \cite{M}, where it was shown that the centralizer algebra $\End_{\mathfrak{p}(n)}(V^{\otimes a})$ is a certain \emph{Brauer superalgebra}, a signed version of the Brauer algebra. One would like to add polynomial generators $y_1,\ldots, y_a$ to the Brauer superalgebra, and define their action on the tensor product $M\otimes V^{\otimes a}$ of an arbitrary $\mathfrak{p}(n)$-representation $M$ with $a$ copies of the vector representation $V$ using an analogue of \eqref{yi-action-Hecke} for some suitably defined element $\Omega\in \mathfrak{p}(n)\otimes  \mathfrak{p}(n)$, which centralizes the action of $\mathfrak{p}(n)$ on tensor products. Unfortunately, such an element {$\Omega$ does not exist in $\mathfrak{p}(n)\otimes  \mathfrak{p}(n)$.

The main idea is to instead consider a {\it fake} Casimir element (see also \cite{BDEHHILNSS}) $$\Omega=\sum_{b\in \cX}{b\otimes b^*}\in \mathfrak{p}(n)\otimes  \mathfrak{gl}(n|n).$$ 
Here $\cX$ is a basis of $\mathfrak{p}(n)$, and $\{b^*\mid b\in \cX\}$ is the dual basis with respect to the supertrace form on $\mathfrak{gl}(n|n)$. This element does not act on a tensor product $M\otimes N$ of arbitrary $\mathfrak{p}(n)$-representations, but does act on the tensor product $M\otimes V$ of an arbitrary $\mathfrak{p}(n)$-representation $M$ and the vector representation $V$ for $\fgl(n|n)$. A formula analogous to \eqref{yi-action-Hecke} defines the action of commuting elements $y_1,\ldots, y_a$ on $M\otimes V^{\otimes a}$, centralizing the $\mathfrak{p}(n)$ action. We thus obtain, see Proposition~\ref{prop:funct}, commuting actions
\begin{eqnarray}
\label{SWpn}
\mathfrak{p}(n)\;\curvearrowright M\otimes V^{\otimes a}\;\curvearrowleft {\sVW}_a, 
\end{eqnarray}
of $\mathfrak{p}(n)$ and a certain {\it affine VW superalgebra} $\sVW_a$. More generally, we establish an action of the affine VW super{\it category} $\sVW$, see Section~\ref{section:sVW}, on the category of modules of the form $M\otimes V^{\otimes a}$ obtained by varying $a$.  Our main result (Theorem~\ref{Thm2}) gives an explicit basis of all the morphism spaces in  $\sVW$. 

The linear independence is proved using the duality \eqref{SWpn} for a specific choice for $M$, namely a Verma module of highest weight $0$. We verify that the PBW filtration on $M$ is compatible with a filtration on the algebras $\sVW_a$, which we build to mimic the filtration by the degree of the polynomials in $\mathbb{C}[y_1,\ldots, y_a]$ in case \eqref{SW}.  We explicitly describe the associated graded algebra and deduce the basis theorem from there. As an application we give a description of the centre of all endomorphism algebras involved. The arguments involve the concept of PBW-deformations and (noncommutative) Rees algebras.

\subsection*{Links to other results of this type}
A special feature of the periplectic Lie superalgebras is that $\sVW_a$ are {\it super}algebras, since the involved endomorphism algebra has odd generators. This does not occur in the context of higher Schur-Weyl dualities of the classical Lie superalgebras (see \cite{CW}, \cite{Sergeev} for a general treatment, \cite{BS5}, \cite{ES2}, \cite{LZ} for different cases with $M=\mathbb{C}$, and \cite{BS4}, \cite{DRV}, \cite{ES1}, \cite{RS}, \cite{Sartori} for higher dualities).  

The superalgebra $\sVW_a$ is a super (or signed) version of the affine VW algebra, defined in \cite{Nazarov} and studied in \cite{ES1} in the context of higher Schur-Weyl dualities for classical Lie algebras  in type $BCD$. In other words, it is a super version of the degenerate BMW algebras, see e.g. \cite{DRV}. This means that, in addition to involving superalgebras, the duality \eqref{SWpn} also has flavours of type $BCD$. In diagrammatic terms, this means working with generalized dotted Brauer diagrams with height moves involving signs.

A basis theorem for the endomorphism algebras of objects in $\sVW$ was obtained independently in \cite{CP} by an algebraic method developed in \cite{Nazarov}, also using the fake Casimir operator. The Brauer superalgebras recently appeared in the literature under the names {\it odd Brauer algebras}, {\it marked Brauer algebras} or {\it periplectic Brauer algebras}, indicating the slightly different points of view on the subject. 

Brauer supercategories can be realized as subcategories, as well as quotients, of the VW supercategories. (In terms of representations, this corresponds to taking $M$ to be the trivial representation; they are a super version of the classical Brauer categories as defined e.g.\ in \cite{LZ2}). As a direct consequence of our basis theorem we thus obtain a basis theorem for the Brauer supercategories, hence reprove results from \cite{BE}, \cite{KT} and \cite{M}. Under this quotient, the elements $y_1,\ldots, y_a$ of the superalgebra $\sVW_a$ specialise to Jucys-Murphy elements in the Brauer superalgebras. This allows one to apply the Cherednik \cite{Cherednik} and Okounkov-Vershik \cite{CST}, \cite{OV}  approaches in this context. First steps in this direction were already successfully taken in \cite{BDEHHILNSS} and \cite{C} from different perspectives to determine the blocks and decomposition numbers in the category of finite dimensional representations of $\mathfrak{p}(n)$ and of the Brauer superalgebra, and further developed in \cite{EC}. A thorough treatment of the corresponding category $\mathcal{O}$ is so far missing and will be deferred to subsequent work. 
 
 \subsection*{The roadmap of the paper}
In Section 2 we define the Brauer supercategory $\sBr$, the VW supercategory $\sVW$, and their endomorphism algebras $\sBr_a$ and $\sVW_a$, and state the main results, Theorems \ref{Thm1} and \ref{Thm2}. In particular, Theorem~\ref{Thm2} gives bases $S_{a,b}^\bullet$ of the endomorphism spaces of $\sVW$. In Section 3 we prove that $S_{a,b}^\bullet$ are spanning sets using  a topological argument. In Section 4 we discuss the Lie superalgebra $\mathfrak{p}(n)$ and its representations, the fake Casimir $\Omega$, and prove the existence of the commuting action \eqref{SWpn}. In Section 5 we prove linear independence of the sets $S_{a,b}^\bullet$ by finding large $n$ and large enough $\mathfrak{p}(n)$-representations $M$, so that the set $S_{a,b}^\bullet$ maps into a set of linearly independent operators on $M\otimes V^{\otimes a}$. This proves Theorem~\ref{Thm2}, and Theorem~\ref{Thm1} follows as a corollary. As an application, in Section~\ref{centre} we describe  the presentation, the centre, and a certain deformation of the endomorphism algebras $\sVW_a=\End_{\sVW}(a)$.

\textbf{Acknowledgements}. We thank Gwyn Bellamy, Michael Ehrig, Stephen Griffeth, Joanna Meinel, Travis Schedler and Anne Shepler for helpful discussions. This project was started at the WINART workshop in Banff, and was developed and finalised during several visits of some of the authors to the Hausdorff Center of Mathematics (in particular to MPI and HIM) in Bonn. We thank these places for the excellent working conditions.

\addtocontents{toc}{\protect\setcounter{tocdepth}{1}}

\section{Definitions and main results}
\label{sectionone}
In this section we define the Brauer supercategory $\sBr$ and the affine VW supercategory $\sVW$ as monoidal supercategories, and state Theorems \ref{Thm1} and \ref{Thm2}, which give diagrammatic bases for the morphism spaces in these categories. 

We fix $\mathbb{C}$ as the ground field for the whole paper. 

\subsection{Monoidal supercategories}

We start by recalling some basic facts about monoidal supercategories. For a thorough discussion, see e.g. \cite{BE}. 

A \emph{superspace} is a vector space $V$ with a $\mathbb{Z}_2$ grading, $V=V_{\overline{0}}\oplus V_{\overline{1}}$. Homogeneous vectors $v\in V_{\overline{0}}$ are said to be \emph{even} or \emph{of parity $\overline{v}=\overline{0}$}, and $v\in V_{\overline{1}}$ are said to be \emph{odd} or \emph{of parity $\overline{v}=\overline{1}$}. Linear maps between superspaces inherit the grading; homogeneous linear maps are called \emph{even} or \emph{odd}, respectively, depending on whether they preserve or change the parity of homogeneous vectors. Formulas involving parity are usually written for homogeneous elements and extended linearly. 
A tensor product of superspaces is again a superspace. For $f,g$ homogeneous linear maps of superspaces, $f\otimes g$ is defined as$$(f\otimes g)(v\otimes w)=(-1)^{\overline{g}\overline{v}} f(v) \otimes g(w)$$
on homogeneous vectors $v \otimes w$.
The following Koszul sign rule holds for compositions
\begin{equation}
\label{signrule}
(f\otimes g)\circ (h\otimes k)= (-1)^{\overline{g}\overline{h}}(f\circ h)\otimes (g\circ k).
\end{equation}
We use the common diagram calculus: the object $a\otimes b$ is depicted by drawing the $b$ to the right of $a$, similar for $f\otimes g$.

A \emph{supercategory} is a category enriched in superspaces; this means all morphism sets are superspaces, and composition preserves parity. 
We will be using the usual string calculus for morphisms in strict monoidal supercategories (see e.g. \cite[Definition XI.2.1]{K}). More precisely, we will define strict monoidal supercategories ($\sBr$ and $\sVW$) using generators and relations by 
\begin{enumerate} [(i)]
\item specifying a set of generating objects; all objects in the category are obtained as finite tensor products $a_1 \otimes \hdots \otimes a_r$ of generating objects $a_i$ (including the empty tensor product, which is defined to be the unit object $\mathbbm{1}$);
\item specifying a set of generating morphisms; all morphisms in the category are then obtained as linear combinations of finite compositions of \emph{horizontal} (using the tensor product $f\otimes g$) and \emph{vertical} (using the composition $f\circ g$) stackings of compatible generating morphisms and the identity morphisms. Diagrammatically, $f \otimes g$ is presented as placing $f$ to the left of $g$, whereas $f \circ g$ is presented as stacking $f$ on top of $g$; in particular, morphisms are read from bottom to top;
\item specifying a set of generating relations for morphisms; the full set of relations is obtained as the two sided tensor ideal generated by the specified generating relations. Implicitly, we also require the morphisms to respect the sign rule (\ref{signrule}); these are sometimes called the \emph{height moves} in string calculus. 
\end{enumerate}

\subsection{The Brauer supercategory $\sBr$} 
 
The \textit{Brauer supercategory} is the $\mathbb{C}$-linear strict monoidal supercategory $\sBr$, generated as a monoidal supercategory by a single object
$\obj$ and morphisms
\vspace{-0.5cm}
\begin{eqnarray*}
&s=\TikZ{[scale=.5] \draw (0,0) node{} to (1,1)node{} (1,0) node{} to (0,1)node{} ;}:\;\obj\otimes \obj\longrightarrow\obj\otimes \obj,&\nonumber \\
&b=\TikZ{[scale=.5] \draw (1,0) node{} (0,0) node{} (0,0) arc(180:0:0.5) ;}:\;\obj\otimes \obj\longrightarrow\mathbbm{1},\quad\text{and}\quad
b^*=\TikZ{[scale=.5] \draw (1,1) node{} (0,1) node{} (0,1) arc(-180:0:0.5) ;}:\;\mathbbm{1}\longrightarrow \obj\otimes \obj ,&
\end{eqnarray*}
with parities $\overline s=\overline 0$, $\overline{b} =\overline{b^*}=\overline 1$, subject to the following defining relations:

\begin{enumerate}[(R1)]
\item {The {\it braid relations}:} $ \TikZ{[scale=.5] \draw (0,0) node{} to (1,1)node{} to (0,2)node{} (1,0) node{} to (0,1)node{}  to (1,2)node{}  ;}= \TikZ{[scale=.5] \draw (0,0) node{} to (0,2)node{} (1,0) node{} to (1,2)node{};}$  and $\TikZ{[scale=.5] \draw  (0,0) node{} to (1,1)node{} (1,0) node{} to (0,1)node{} (2,0) node{} to (2,1)node{} (0,1) node{} to (0,2)node{} (1,1) node{} to (2,2)node{} (2,1) node{} to (1,2)node{} (0,2) node{} to (1,3)node{} (1,2) node{} to (0,3)node{} (2,2) node{} to (2,3)node{};}=
\TikZ{[scale=.5] \draw (0,0) node{} to (0,1)node{} (1,0) node{} to (2,1)node{} (2,0) node{} to (1,1)node{}
(0,1) node{} to (1,2)node{} (1,1) node{} to (0,2)node{} (2,1) node{} to (2,2)node{} 
(0,2) node{} to (0,3)node{} (1,2) node{} to (2,3)node{} (2,2) node{} to (1,3)node{}
;}$,  
\item {The {\it snake relations} or {\it adjunctions}:} $ \TikZ{[scale=.5] \draw
(0,0) node{} to (0,1)node{} 
(0,1) arc(180:0:0.5)
(1,1) arc(-180:0:0.5)
(2,1)node{} to (2,2)node{};}=
-\TikZ{[scale=.5] \draw
(0,0) node{} to (0,2)node{} 
;}$ and $ \TikZ{[scale=.5] \draw
(0,1) node{} to (0,2)node{} 
(1,1) arc(180:0:0.5)
(0,1) arc(-180:0:0.5)
(2,0)node{} to (2,1)node{};}=
 \TikZ{[scale=.5] \draw
(0,0) node{} to (0,2)node{} 
;}$, 
\item {The {\it untwisting} relations}:  $\TikZ{[scale=.5] \draw
(0,1) node{} to (0,2)node{} 
(2,0.5) node{} to (2,1)node{} 
(1,1) node{} to (2,2)node{} 
(2,1) node{} to (1,2)node{} 
(0,1) arc(-180:0:0.5)
;}= 
\TikZ{[scale=.5] \draw
(0,0.5) node{} to (0,1)node{} 
(2,1) node{} to (2,2)node{} 
(0,1) node{} to (1,2)node{} 
(1,1) node{} to (0,2)node{} 
(1,1) arc(-180:0:0.5)
;}$ and  $ \TikZ{[scale=.5] \draw
(1,1) node{} to (0,2)node{} 
(0,1) node{} to (1,2)node{} 
(0,1) arc(-180:0:0.5)
;}= -
\TikZ{[scale=.5] \draw
(1,1) node{} to (1,2)node{} 
(0,1) node{} to (0,2)node{} 
(0,1) arc(-180:0:0.5)
;}$. 
\end{enumerate}
The supercategory structure means the {\it height moves} via \eqref{signrule} are also satisfied, e.g. 
$\TikZ{[scale=.5] \draw
	(0,0) node{} to (0,0.5)node{} 
	(0,0.5) arc(180:0:0.5)
	(1,0) node{} to (1,0.5) node{}
	(1.5,0) arc(180:0:0.5);}=b \circ (1 \otimes 1 \otimes b) = b \otimes b =
\TikZ{[scale=.5] \draw
	(0,0) arc(180:0:0.5)
	(1.5,0) arc(180:0:0.5);} \; , \quad 
\TikZ{[scale=.5] \draw
	(1.5,0) node{} to (1.5,0.5)node{} 
	(1.5,0.5) arc(180:0:0.5)
	(2.5,0) node{} to (2.5,0.5) node{}
	(0,0) arc(180:0:0.5);}=b \circ (b \otimes 1 \otimes 1) =-\; b \otimes b =
-\;\TikZ{[scale=.5] \draw
	(0.0,0) arc(180:0:0.5)
	(1.5,0) arc(180:0:0.5);}.$

 The objects of $\sBr$ are sometimes written as natural numbers $\mathbb{N}_0$, identifying $a\in\mathbb{N}_0$ with $\obj^{\otimes a}$, where $\obj^{\otimes 0}=\mathbbm{1}$. 
 A {\it diagram} is a finite composition (horizontally or vertically) of generating morphisms and identity morphisms. It consists of lines, connecting pairs of points among the bottom and top ones, which we call \emph{strings}. Elements of $\Hom_{\sBr}(a,b)$ are linear combinations of diagrams with strings connecting $a$ points at the bottom and $b$ points at the top. We let $1_a\in \Hom_{\sBr}(a,a)$ denote the identity morphism, and let
\begin{eqnarray*}
&b_i=1_{i-1}\otimes b\otimes 1_{a-i+1} \in \Hom_{\sBr}(a+2,a), \quad \quad
b_i^*=1_{i-1}\otimes b^*\otimes 1_{a-i+1} \in \Hom_{\sBr}(a,a+2),\\ 
&s_i=1_{i-1}\otimes s\otimes 1_{a-i-1} \in \Hom_{\sBr}(a,a)
\end{eqnarray*}
denote the morphisms obtained by applying $b,b^*$ and $s$ on the $i$-th and $(i+1)$-st tensor factors. The supercategory $\sBr$ can alternatively be generated as a supercategory (as opposed to a monoidal supercategory) by vertically stacking compatible $b_i,b_i^*,s_i$.
 
\subsection{Normal diagrams} 
 We call a string with both ends at the top of the diagram a \emph{cup}, a string with both ends at the bottom of the diagram a \emph{cap}, a string with one end at the top and one at the bottom a \emph{through string}, and a string with no endpoints a \emph{loop}.

Call a diagram $d\in \Hom_{\sBr}(a,b)$ \emph{normal} if all of the following hold:
\begin{itemize}
\item any two strings intersect at most once;
\item no string intersects itself; 
\item no two cups or caps are at the same height; 
\item all cups are above all caps;
\item the height of caps decreases when the caps are ordered from left to right with respect to their left ends;
\item the height of cups increases when the caps are ordered from left to right with respect to their left ends.
\end{itemize}
 As a consequence, every string in a normal diagram has either one cup, or one cap, or no cups and caps, and there are no closed loops. A diagram with no loops in $\Hom_{\sBr}(a,b)$ has $\frac{a+b}{2}$ strings. In particular, if $a+b$ is odd then this space is zero.
 
Each normal diagram $d \in \Hom_{\sBr}(a,b)$, where $a,b\in \mathbb{N}_0$,  gives rise to a partition $P(d)$ of the set of $a+b$ points into $2$-element subsets given by the endpoints of the strings in $d$. We call such a partition a \emph{connector} and let $\Conn(a,b)$ denote the set of all such connectors; its size is $(a+b-1)!!$. For each connector $c\in \Conn(a,b)$, we pick a normal diagram $d_c\in P^{-1}(c) \subset \Hom_{\sBr}(a,b)$. (Note that different normal diagrams in a single fibre $P^{-1}(c)$ differ only by braid relations, and thus represent the same morphism, see Lemma~\ref{choice of Sab is a choice of sign}.)

\begin{theorem}\label{Thm1}
The set $S_{a,b}=\{ d_c  \mid  c\in \Conn(a,b)\}$ is a basis of $\Hom_{\sBr}(a,b)$.\end{theorem} 
We show that it is a spanning set using topology in Section~\ref{Sect-span}. Linear independence can also be seen directly using topology, since the defining relations of $\sBr$ do not change the underlying connector of a diagram. However, we obtain it using representation theory in Section~\ref{linindep} as a direct consequence of the more general Theorem~\ref{Thm2}. For the special case of $a=b$, this theorem appears as a basis theorem for the algebra $\mathcal{A}_a$ in \cite{M}.

Let us also remark that the above choice of normal diagrams for basis vectors is for convenience only. It is enough to choose one diagram $d'_c$ with no loops in every fibre $P^{-1}(c)$; the set $\{ d'_c \mid c\in \Conn(a,b) \}$ is then also a basis. This choice of basis differs from $S_{a,b}$ by signs only, meaning it is a subset of $\{ \pm d \mid d \in \Conn(a,b)\}$ with exactly one choice of sign for each $d$, see Proposition~\ref{oBr-span}.

\subsection{The affine VW supercategory $\sVW$} \label{section:sVW}

The \emph{affine VW supercategory}, or \emph{affine Nazarov-Wenzl supercategory}, is the $\mathbb{C}$-linear strict monoidal supercategory $\sVW$, generated as a monoidal supercategory by a single object
$\obj$, morphisms $s=\TikZ{[scale=.5] \draw (0,0) node{} to (1,1)node{} (1,0) node{} to (0,1)node{} ;}:\;\obj\otimes \obj\longrightarrow\obj\otimes \obj$, $b=\TikZ{[scale=.5] \draw (1,0) node{} (0,0) node{} (0,0) arc(180:0:0.5) ;}:\;\obj\otimes \obj\longrightarrow\mathbbm{1}$ and $b^*=\TikZ{[scale=.5] \draw (1,1) node{} (0,1) node{} (0,1) arc(-180:0:0.5) ;}:\;\mathbbm{1}\longrightarrow \obj\otimes \obj$ as above, and an additional morphism
\begin{eqnarray*}
&y=\TikZ{[scale=.5] \draw (0,0) node{} to (0,1) node{} (0,0.5) node[fill,circle,inner sep=1.5pt]{};}:\;\obj \longrightarrow \obj&
\end{eqnarray*}
with $\overline y=0$, subject to relations (R1)-(R3) above, and 

\begin{enumerate}
\item[(R4)]  {The {\it dot relations}:} $\TikZ{[scale=.5] \draw
(0,0) node{} to (0,2) node{} 
(1,0) node{} to (1,2) node{} (1,1) node[fill,circle,inner sep=1.5pt]{} 
;} =
\TikZ{[scale=.5] \draw
(0,0) node{} to (1,1) node{}  to (0,2) node{} 
(1,0) node{} to (0,1) node{}  to (1,2) node{} (0,1) node[fill,circle,inner sep=1.5pt]{} 
;}+
\TikZ{[scale=.5] \draw
(0,0) node{} to (1,2) node{} (1,0) node{} to (0,2) node{} 
;}+
\TikZ{[scale=.5] \draw
(0,0) arc(180:0:0.5) (0,2) arc(-180:0:0.5) 
;}$ \textrm{ and }
$\TikZ{[scale=.5] \draw 
(0,1) arc(180:0:0.5) (0,0) node{} to (0,1) node{} (1,0) node{} to (1,1) node{} (1,0.5) node[fill,circle,inner sep=1.5pt]{} 
;}=
\TikZ{[scale=.5] \draw 
(0,1) arc(180:0:0.5) (0,0) node{} to (0,1) node{} (1,0) node{} to (1,1) node{} (0,0.5) node[fill,circle,inner sep=1.5pt]{} 
;}+
\TikZ{[scale=.5] \draw 
(0,1) arc(180:0:0.5) (0,0) node{} to (0,1) node{} (1,0) node{} to (1,1) node{} 
;}.
$
\end{enumerate}

The objects in $\sVW$ can be identified with integers $a\in \mathbb{N}_0$, and the morphisms are linear combinations of dotted diagrams. The category can alternatively be generated by vertically stacking $b_i,b_i^*,s_i$ and $y_i=1_{i-1}\otimes y\otimes 1_{a-i}\in \Hom_{\sVW}(a,a)$. It is a filtered category, in the sense that the spaces $\Hom_{\sVW}(a,b)$ have a filtration with $\Hom_{\sVW}(a,b)^{\le k }$ being the span of all dotted diagrams with at most $k$ dots.

\subsection{Normal dotted diagrams}
Call a dotted diagram $d\in \Hom_{\sVW}(a,b)$ \emph{normal} if:
\begin{itemize}
\item the underlying diagram obtained by erasing the dots is normal; 
\item all dots on cups and caps are on the leftmost end, and all dots on the through strings are at the bottom. 
\end{itemize}

Let $S^\bullet _{a,b}$ be the set of normal dotted diagrams obtained by taking all diagrams in $S_{a,b}$ and adding dots to them in all possible ways. Let $S^{k} _{a,b}\subset S^\bullet _{a,b}$ and $S^{\le k}_{a,b}=\bigcup _{l=0}^k S^l_{a,b}$ be the sets of such diagrams with exactly $k$ dots, respectively at most $k$ dots. In particular, $S^{0}_{a,b}=S^{\leq 0}_{a,b}=S_{a,b}$. Note that if $a\equiv b \text{ mod } 2$ then the cardinality of $S^k_{a,b}$ is ${\frac{a+b}{2}+k-1 \choose k}\cdot (a+b-1)!!$, and if $ a \not\equiv b \text{ mod } 2$ then the cardinality of $S^k_{a,b}$ is $0$.

The following basis theorem is the main result of this paper. 
 
\begin{theorem}[Basis Theorem]\label{Thm2}
The set $S^{\le k} _{a,b}$ is a basis of $\Hom_{\sVW}(a,b)^{\le k}$, and consequently the set $S^\bullet _{a,b}$ is a basis of $\Hom_{\sVW}(a,b)$.      
\end{theorem} 

The proof will be given in Sections~\ref{Sect-span} and \ref{linindep}. The identification $S_{a,b}=S^0_{a,b}$ defines an embedding of categories $\sBr \longrightarrow \sVW$ and hence Theorem~\ref{Thm2} directly implies Theorem~\ref{Thm1}. 

As an immediate consequence of Theorem~\ref{Thm2} we obtain the following:

\begin{corollary}
The diagrams without dots form a supersubalgebra $\Hom_{\sBr}(a,a)$  of the superalgebra  $\Hom_{\sVW}(a,a)$. The dotted diagrams whose underlying undotted diagram is the identity morphism $1_a$ form a polynomial subalgebra $\C[y_1,\ldots,y_a]$, and the subalgebras $\C[y_1,\ldots,y_a]$ and $\Hom_{\sBr}(a,a)$ together generate $\Hom_{\sVW}(a,a)$ as vector superalgebra. 
\end{corollary}

\subsection{The affine VW superalgebra $\sVW_a$}
	
For any $a \in \NN$, the endomorphism space $\sVW_a=\Hom_{\sVW}(a,a)$ has the structure of a superalgebra. It is the signed version of the affine VW algebra (see \cite[Section 2]{ES1} for the setup we use), and the affine version of the Brauer superalgebra  $\Hom_{\sBr}(a,a)$. These algebras have an interesting structure, and allow an $\hbar$-deformation. For more details, including a presentation and a description of the centre, see Section \ref{centre}.

One can also define {\it cyclotomic quotients} of the algebras $\sVW_a$ by 
mimicking the constructions in \cite{AMR} for affine VW algebras, see also \cite{CP}.  We expect Lemma \ref{lem:dot-loop} (stating the vanishing of all loop values) to simplify the necessary admissibility conditions from \cite{AMR} and more explicitly \cite{ES1} drastically, but do not pursue this here.

\section{Spanning sets for $\sBr$ and $\sVW$}\label{Sect-span}

In this section we show that the sets $S_{a,b}$ and $S^{\bullet}_{a,b}$ span the corresponding morphism spaces in the categories $\sBr$ and $\sVW$ (Propositions \ref{oBr-span} and \ref{oVW-span}). 

\subsection{Some diagrammatic relations}

First, we establish some additional relations in these categories. Note that these relations are local and hold wherever they are defined within a bigger expression, and we indicate how the local diagram fits into the larger one by specifying the position ($i \in \NN$) of a string (always counted from the left).

The first lemma shows that in $\sBr$ (and consequently in $\sVW$), similar untwisting relations to (R3) hold for caps as they do for cups, and that any isolated loops are zero.

\begin{lemma}[Untwisting relations]
	\label{lem:sBRrel} The following relations hold in $\sBr$ and $\sVW$:
	\begin{align*}
	(a) \quad
	\TikZ{[scale=.5] \draw
		(0,0) node{} to (0,1)node{} 
		(2,1) node{} to (2,1.5)node{} 
		(1,0) node{} to (2,1)node{} 
		(1,1) node{} to (2,0)node{} 
		(0,1) arc(180:0:0.5)
		;}&= 
	\TikZ{[scale=.5] \draw
		(0,1) node{} to (0,1.5)node{} 
		(2,0) node{} to (2,1)node{} 
		(0,0) node{} to (1,1)node{} 
		(0,1) node{} to (1,0)node{} 
		(1,1) arc(180:0:0.5)
		;} \hspace{0.8cm} &
	(b) \quad  
	\TikZ{[scale=.5] \draw
		(1,0) node{} to (0,1)node{} 
		(0,0) node{} to (1,1)node{} 
		(0,1) arc(180:0:0.5)
		;}&= 
	\TikZ{[scale=.5] \draw
		(1,0) node{} to (1,1)node{} 
		(0,0) node{} to (0,1)node{} 
		(0,1) arc(180:0:0.5)
		;}\hspace{0.8cm} &
	(c) \quad 
	\TikZ{[scale=.5] \draw
		(0,1) arc(0:360:0.5)
		;}&=0 \\
	b_{i} s_{i+1}&=b_{i+1} s_{i}   &  b_{i} s_{i}&=b_{i}  &    b_{i} b_{i}^*&=0
	\end{align*} 
\end{lemma}

\begin{proof}
	\begin{enumerate}[(a)]
		\item Using the relations in $\sBr$ and \eqref{signrule}, the morphism $s$ can be rewritten as
		$$\TikZ{[scale=.5] \draw
			(0,0) node{} to (1,1)node{}
			(1,0) node{} to (0,1)node{}
			;}\stackrel{(R2)}{=}
		-\TikZ{[scale=.5] \draw
			(0,0.5) node{} to (0,1)node{} 
			(0,1) arc(180:0:0.5)
			(1,1) arc(-180:0:0.5)
			(2,1)node{} to (2,1.5)node{} 
			(2,1.5)node{} to (3,2.5)node{}
			(3,1.5)node{} to (2,2.5)node{}
			(3,0.5)node{} to (3,1.5)node{}
			;}=
		-\TikZ{[scale=.5] \draw
			(0,0.5) node{} to (0,2)node{} 
			(0,2) arc(180:0:0.5)
			(1,2)node{} to (1,1)node{}
			(1,1) arc(-180:0:0.5)
			(2,2)node{} to (2,2.5)node{}
			(2,1)node{} to (3,2)node{}
			(3,1)node{} to (2,2)node{}
			(3,2)node{} to (3,2.5)node{}
			(3,0.5)node{} to (3,1)node{}
			;}\stackrel{(R3)}{=}
		-\TikZ{[scale=.5] \draw
			(0,0.5) node{} to (0,2)node{} 
			(0,2) arc(180:0:0.5)
			(2,2.5)node{} to (2,2)node{}
			(2,1) arc(-180:0:0.5)
			(1,1)node{} to (2,2)node{}
			(2,1)node{} to (1,2)node{}
			(1,0.5)node{} to (1,1)node{}
			(3,1)node{} to (3,2.5)node{}
			;}
		$$
		and therefore
		$$
		\TikZ{[scale=.5] \draw
			(0,0) node{} to (1,1)node{}
			(1,0) node{} to (0,1)node{}
			(0,1)node{} to (0,1.5)node{}
			(1,1) arc(180:0:0.5)
			(2,0)node{} to (2,1)node{}
			;}=
		-\TikZ{[scale=.5] \draw
			(0,0.5) node{} to (0,2)node{} 
			(0,2) arc(180:0:0.5)
			(2,2)node{} to (2,3)node{}
			(2,1) arc(-180:0:0.5)
			(1,1)node{} to (2,2)node{}
			(2,1)node{} to (1,2)node{}
			(1,0.5)node{} to (1,1)node{}
			(3,1)node{} to (3,2.5)node{}
			(3,2.5) arc(180:0:0.5)
			(4,0.5)node{} to (4,2.5)node{}
			;}=
		\TikZ{[scale=.5] \draw
			(0,0) node{} to (0,2)node{} 
			(0,2) arc(180:0:0.5)
			(1,2)node{} to (2,1)node{}
			(2,1)node{} to (2,0.5)node{}
			(2,0.5) arc(-180:0:0.5)
			(3,0.5) arc(180:0:0.5)
			(4,0.5)node{} to (4,0)node{}
			(2,2.5)node{} to (2,2)node{}
			(2,2)node{} to (1,1)node{}
			(1,1)node{} to (1,0)node{}
			;}\stackrel{(R2)}{=}
		\TikZ{[scale=.5] \draw
			(0,0) node{} to (0,1)node{} 
			(0,1) arc(180:0:0.5)
			(2,1)node{} to (2,1.5)node{}
			(1,0)node{} to (2,1)node{}
			(2,0)node{} to (1,1)node{}
			;}
		.$$

		\item We use part (a), the relations in $\sBr$ and the Koszul sign rule \eqref{signrule} to show
		\begin{align*}\TikZ{[scale=.5] \draw
			(0,0) node{} to (1,1)node{} 
			(1,1) arc(-180:0:-0.5)
			(0,1)node{} to (1,0)node{};}&\stackrel{(R2)}{=}
		\TikZ{[scale=.5] \draw
			(0,-1.5)node{} to (0,-0.5)node{}
			(0,-0.5) node{} to (1,0.5)node{} 
			(1,0.5) arc(-180:0:-0.5)
			(0,0.5)node{} to (1,-0.5)node{}
			(1,-0.5) node{} to (1,-1)node{} 
			(1,-1) arc(-180:0:0.5)
			(2,-1) arc(180:0:0.5)
			(3,-1)node{} to (3,-1.5)node{};}=
		-\TikZ{[scale=.5] \draw
			(0,-0.5)node{} to (0,0)node{}
			(0,0) node{} to (1,1)node{} 
			(1,1) arc(-180:0:-0.5)
			(0,1)node{} to (1,0)node{}
			(1,0) arc(-180:0:0.5)
			(2,0)node{} to (2,1.5)node{}
			(2,1.5) arc(180:0:0.5)
			(3,1.5)node{} to (3,-0.5)node{};}\stackrel{(R3)}{=}
		-\TikZ{[scale=.5] \draw
			(0,0.5) node{} to (0,1)node{}
			(0,1)node{} to (-1,2)node{}
			(-1,2) arc(-180:0:-0.5)
			(-2,2)node{} to (-2,1)node{}
			(-2,1) arc(-180:0:0.5)
			(-1,1)node{} to (0,2)node{}
			(0,2)node{} to (0,2.5)node{}
			(0,2.5) arc(180:0:0.5)
			(1,2.5)node{} to (1,0.5)node{};}=\\
		&\stackrel{(a)}{=}
		-\TikZ{[scale=.5] \draw
			(2,0) node{} to (2,1.5)node{}
			(2,1.5) arc(-180:0:-0.5)
			(1,1.5)node{} to (0,0.5)node{}
			(0,0.5) arc(-180:0:0.5)
			(1,0.5)node{} to (0,1.5)node{}
			(0,1.5) arc(180:0:1.5)
			(3,1.5)node{} to (3,0)node{};}\stackrel{(R3)}{=}
		\TikZ{[scale=.5] \draw
			(2,0) node{} to (2,1.5)node{}
			(2,1.5) arc(-180:0:-0.5)
			(1,1.5)node{} to (1,0.5)node{}
			(0,0.5) arc(-180:0:0.5)
			(0,0.5)node{} to (0,1.5)node{}
			(0,1.5) arc(180:0:1.5)
			(3,1.5)node{} to (3,0)node{};}\stackrel{(R2)}{=}	
		\TikZ{[scale=.5] \draw
			(0,0)node{} to (0,1)node{}
			(0,1) arc(180:0:0.5)
			(1,1)node{} to (1,0)node{};}
		\end{align*}

		\item With $\TikZ{[scale=.5] \draw
			(0,0)node{} to (0,1)node{}
			(1,0)node{} to (1,1)node{}
			(0,1) arc(180:0:0.5)
			;}=
		\TikZ{[scale=.5] \draw
			(0,0)node{} to (1,1)node{} (0,1)node{} to (1,0)node{}
			(0,1) arc(180:0:0.5)
			;}
		$ and $-\TikZ{[scale=.5] \draw
			(0,1)node{} to (0,0)node{}
			(1,1)node{} to (1,0)node{}
			(0,0) arc(-180:0:0.5);}=
		\TikZ{[scale=.5] \draw
			(0,1) arc(-180:0:0.5)
			(0,1)node{} to (1,2)node{} (1,1)node{} to (0,2)node{}
			;}$, we have
		$
		\TikZ{[scale=.5] \draw
			(0,1) arc(360:0:0.5)
			;}=
		\frac{1}{2} \TikZ{[scale=.5] \draw
			(0,1) arc(360:0:0.5)
			;}+\frac{1}{2} \TikZ{[scale=.5] \draw
			(0,1) arc(360:0:0.5)
			;}=
		\frac{1}{2} \TikZ{[scale=.5] \draw
			(0,1) arc(-180:0:0.5)
			(0,1)node{} to (1,2)node{} (0,2)node{} to (1,1)node{}
			(0,2) arc(180:0:0.5)
			;}-\frac{1}{2} \TikZ{[scale=.5] \draw
			(0,1) arc(-180:0:0.5)
			(0,1)node{} to (1,2)node{} (0,2)node{} to (1,1)node{}
			(0,2) arc(180:0:0.5)
			;}=0.
		$	\qedhere
	\end{enumerate}
\end{proof}

The next lemma explains how a dot can be moved within a dotted diagram in $\sVW$. In particular, it can slide through crossings and cups, modulo some diagrams with a smaller number of dots.

\begin{lemma}[Dot sliding relations]\label{lem:dot-slide} The following relations hold in $\sVW$:
	\begin{align*}
	(a) \;\;\TikZ{[scale=.5] \draw 
		(0,0) node{} to (1,2) node{} (1,0) node{} to (0,2) node{} (0.75,0.5) node[fill,circle,inner sep=1.5pt]{} 
		;}&=
	\TikZ{[scale=.5] \draw 
		(0,0) node{} to (1,2) node{} (1,0) node{} to (0,2) node{} (0.25,1.5) node[fill,circle,inner sep=1.5pt]{} 
		;}+
	\TikZ{[scale=.5] \draw 
		(0,0) node{} to (0,2) node{} (1,0) node{} to (1,2) node{} ;}-
	\TikZ{[scale=.5] \draw 
		(0,0) arc(180:0:0.5) (0,2) arc(-180:0:0.5)
		;}
	&
	(b) \;\;\TikZ{[scale=.5] \draw 
		(0,0) node{} to (1,2) node{} (1,0) node{} to (0,2) node{} (0.25,0.5) node[fill,circle,inner sep=1.5pt]{} 
		;}&=
	\TikZ{[scale=.5] \draw 
		(0,0) node{} to (1,2) node{} (1,0) node{} to (0,2) node{} (0.75,1.5) node[fill,circle,inner sep=1.5pt]{} 
		;}-
	\TikZ{[scale=.5] \draw 
		(0,0) node{} to (0,2) node{} (1,0) node{} to (1,2) node{} ;}-
	\TikZ{[scale=.5] \draw 
		(0,0) arc(180:0:0.5) (0,2) arc(-180:0:0.5)
		;} &
	(c) \;\;\TikZ{[scale=.5] \draw 
		(0,1) arc(-180:0:0.5)
		(0,1) node{} to (0,2) node{} (1,1) node{} to (1,2) node{} (1,1.5) node[fill,circle,inner sep=1.5pt]{} 
		;}&=\TikZ{[scale=.5] \draw 
		(0,1) arc(-180:0:0.5) (0,1) node{} to (0,2) node{} (1,1) node{} to (1,2) node{} (0,1.5) node[fill,circle,inner sep=1.5pt]{} 
		;}-
	\TikZ{[scale=.5] \draw 
		(0,1) arc(-180:0:0.5) (0,1) node{} to (0,2) node{} (1,1) node{} to (1,2) node{} 
		;}\\
	s_iy_{i+1}&=y_is_i+1-b^*_ib_i   &   s_iy_i&=y_{i+1}s_i-1-b_i^*b_i & y_{i+1}b^*_i&=y_ib^*_i-b^*_i
	\end{align*}
\end{lemma}

\begin{proof}
	To obtain the relations (a) and (b), we multiply the first relation in (R4) by $s_i$ on the left, respectively on the right, and then use the braid and untwisting relations (R1), (R3) together with Lemma~\ref{lem:sBRrel}{(b)} to simplify.
	To prove (c), we compute:
	$$\TikZ{[scale=.5] \draw 
		(0,1) arc(-180:0:0.5)
		(0,1)node{} to (0,2)node{}
		(1,1)node{} to (1,2)node{}
		(1,1.5) node[fill,circle,inner sep=1.5pt]{} 
		;}\stackrel{(R2)}{=}
	-\TikZ{[scale=.5] \draw
		(0,3)node{} to (0,1)node{}
		(0,1) arc(-180:0:0.5)
		(1,1)node{} to (1,2.5)node{}
		(1,1.5) node[fill,circle,inner sep=1.5pt]{} 
		(1,2.5) arc(180:0:0.5)
		(2,2.5) arc(-180:0:0.5)
		(3,2.5)node{} to (3,3)node{}
		;}=
	\TikZ{[scale=.5] \draw
		(0,2.5)node{} to (0,1)node{}
		(0,1) arc(-180:0:0.5)
		(1,1)node{} to (1,2)node{}
		(1,1.5) node[fill,circle,inner sep=1.5pt]{} 
		(1,2) arc(180:0:0.5)
		(2,2)node{} to (2,0)node{}
		(2,0) arc(-180:0:0.5)
		(3,0)node{} to (3,2.5)node{}
		;}\stackrel{(R4)}{=}
	\TikZ{[scale=.5] \draw
		(0,2.5)node{} to (0,1)node{}
		(0,1) arc(-180:0:0.5)
		(1,1)node{} to (1,2)node{}
		(2,1.5) node[fill,circle,inner sep=1.5pt]{} 
		(1,2) arc(180:0:0.5)
		(2,2)node{} to (2,0)node{}
		(2,0) arc(-180:0:0.5)
		(3,0)node{} to (3,2.5)node{}
		;}-
	\TikZ{[scale=.5] \draw
		(0,2.5)node{} to (0,1)node{}
		(0,1) arc(-180:0:0.5)
		(1,1)node{} to (1,2)node{}
		(1,2) arc(180:0:0.5)
		(2,2)node{} to (2,0)node{}
		(2,0) arc(-180:0:0.5)
		(3,0)node{} to (3,2.5)node{}
		;}\stackrel{(R2)}{=}
	\TikZ{[scale=.5] \draw 
		(0,1) arc(-180:0:0.5)
		(0,1) node{} to (0,2) node{}
		(1,1) node{} to (1,2) node{}
		(0,1.5) node[fill,circle,inner sep=1.5pt]{} 
		;}-
	\TikZ{[scale=.5] \draw 
		(0,1) arc(-180:0:0.5)
		(0,1) node{} to (0,2) node{}
		(1,1) node{} to (1,2) node{} 
		;}.
	\hfill\qedhere$$
\end{proof}

By induction, we obtain formulas for sliding dots along cups or caps:

\begin{lemma}\label{lem:dot-on-cup} The following relations hold in $\sVW$ for any $k\geq 1$.
\begin{align*}
 	&(a) \quad \TikZ{[scale=.5] \node [right] at (1,0.5) {k};
		\draw 
		(0,1) arc(180:0:0.5) (0,0) node{} to (0,1) node{} (1,0) node{} to (1,1) node{} (1,0.5) node[fill,circle,inner sep=1.5pt]{} 
		;}=\sum\limits_{j=0}^k{k\choose j}\TikZ{[scale=.5] \node[left] at (0,0.5) {j};
		\draw 
		(0,1) arc(180:0:0.5) (0,0) node{} to (0,1) node{} (1,0) node{} to (1,1) node{} (0,0.5) node[fill,circle,inner sep=1.5pt]{} 
		;},
	&(b) \quad \TikZ{[scale=.5] \node [left] at (0,0.5) {k};
		\draw 
		(0,1) arc(180:0:0.5) 
		(0,0) node{} to (0,1) node{} 
		(1,0) node{} to (1,1) node{} 
		(0,0.5) node[fill,circle,inner sep=1.5pt]{} 
		;}=\sum\limits_{j=0}^k (-1)^{k+j}{k\choose j} 
	\TikZ{[scale=.5] 
		\node[right] at (1,0.5) {j};
		\draw 
		(0,1) arc(180:0:0.5) 
		(0,0) node{} to (0,1) node{} 
		(1,0) node{} to (1,1) node{} 
		(1,0.5) node[fill,circle,inner sep=1.5pt]{} 
		;}, \\
	\vspace{1cm}
	&(c) \quad \TikZ{[scale=.5] \node[left] at (0,1.5){k};
		\draw 
		(0,1) arc(-180:0:0.5) (0,1) node{} to (0,2) node{} (1,1) node{} to (1,2) node{} (0,1.5) node[fill,circle,inner sep=1.5pt]{} 
		;}=\sum\limits_{j=0}^k{k\choose j}\TikZ{[scale=.5] \node[right] at (1,1.5){j};
		\draw 
		(0,1) arc(-180:0:0.5) (0,1) node{} to (0,2) node{} (1,1) node{} to (1,2) node{} (1,1.5) node[fill,circle,inner sep=1.5pt]{} 
		;},
    &(d) \quad \TikZ{[scale=.5] 
		\node[right] at (1,1.5){k};
		\draw 
		(0,1) arc(-180:0:0.5) 
		(0,1) node{} to (0,2) node{} 
		(1,1) node{} to (1,2) node{} 
		(1,1.5) node[fill,circle,inner sep=1.5pt]{} 
		;}=\sum\limits_{j=0}^k (-1)^{k+j}{k\choose j} 
	\TikZ{[scale=.5] 
		\node[left] at (0,1.5){j};
		\draw 
		(0,1) arc(-180:0:0.5) 
		(0,1) node{} to (0,2) node{} 
		(1,1) node{} to (1,2) node{} 
		(0,1.5) node[fill,circle,inner sep=1.5pt]{} 
		;},
\end{align*}
where the integers attached to the dots indicate the number of dots on the strand.
\end{lemma}

The following formulas for sliding dots through a crossing can also be verified in a straightforward way using induction, and should be compared with \cite[Lemma 2.3]{AMR}.

\begin{lemma}[Generalized dot sliding]\label{AMR2.3} For any $k\in\mathbb{Z}_{\geq0}$ we have the following relations:
		\begin{align*}
		(a) \;\;\TikZ{[scale=.5] 
			\node[right] at (0.75,0.5){k};
			\draw 
			(0,0) node{} to (1,2) node{} (1,0) node{} to (0,2) node{} (0.75,0.5) node[fill,circle,inner sep=1.5pt]{} 
			;}&=
		\TikZ{[scale=.5]
			\node[left] at (0.25,1.5){k};
			\draw 
			(0,0) node{} to (1,2) node{} (1,0) node{} to (0,2) node{}
			(0.25,1.5) node[fill,circle,inner sep=1.5pt]{} 
			;}+\sum^{k-1}_{j=0}\left(
		\TikZ{[scale=.5]
			\node[left] at (0,1.5){k-1-j};
			\node[right] at (1,0.5){j};
			\draw 
			(0,0) node{} to (0,2) node{} (1,0) node{} to (1,2) node{}
			(0,1.5) node[fill,circle,inner sep=1.5pt]{}
			(1,0.5) node[fill,circle,inner sep=1.5pt]{}  ;}-
		\TikZ{[scale=.5]
			\node[left] at (0.12,1.71){k-1-j};
			\node[right] at (0.88,0.3){j};
			\draw 
			(0,0) arc(180:0:0.5) (0,2) arc(-180:0:0.5)
			(0.12,1.71) node[fill,circle,inner sep=1.5pt]{}
			(0.88,0.3) node[fill,circle,inner sep=1.5pt]{} 
			;}\right) \\
		(b) \;\;\TikZ{[scale=.5]
			\node[left] at (0.25,0.5){k};
			\draw 
			(0,0) node{} to (1,2) node{} (1,0) node{} to (0,2) node{}
			(0.25,0.5) node[fill,circle,inner sep=1.5pt]{} 
			;}&=
		\TikZ{[scale=.5]
			\node[right] at (0.75,1.5){k};
			\draw 
			(0,0) node{} to (1,2) node{} (1,0) node{} to (0,2) node{}
			(0.75,1.5) node[fill,circle,inner sep=1.5pt]{} 
			;}-\sum^{k-1}_{j=0}\left(
		\TikZ{[scale=.5]
			\node[left] at (0,0.5){k-1-j};
			\node[right] at (1,1.5){j};
			\draw 
			(0,0) node{} to (0,2) node{} (1,0) node{} to (1,2) node{}
			(0,0.5) node[fill,circle,inner sep=1.5pt]{}
			(1,1.5) node[fill,circle,inner sep=1.5pt]{}  ;}+
		\TikZ{[scale=.5]
			\node[left] at (0.12,0.3){k-1-j};
			\node[right] at (0.88,1.71){j};
			\draw 
			(0,0) arc(180:0:0.5) (0,2) arc(-180:0:0.5)
			(0.12,0.3) node[fill,circle,inner sep=1.5pt]{}
			(0.88,1.71) node[fill,circle,inner sep=1.5pt]{} 
			;}\right)\end{align*}
\end{lemma}

Furthermore, as we show next, as a generalization of Lemma~\ref{lem:sBRrel}{(c)}, isolated loops in $\sVW$ with any number of dots are zero.

\begin{lemma}[Loop values] \label{lem:dot-loop}
For any $k,\ell\in \mathbb{N}_0$, the following relation holds in $\sVW$:
	
	\begin{equation*}
	\TikZ{[scale=.5]
		\node [left] at (0,1.5) {k};
		\node [right] at (1,1.5) {$\ell$};
		\draw 
		(0,1) arc(-180:0:0.5) (0,2) arc(180:0:0.5)
		(0,1) node{} to (0,2) node{} (1,1) node{} to (1,2) node{}
		(0,1.5) node[fill,circle,inner sep=1.5pt]{}  (1,1.5) node[fill,circle,inner sep=1.5pt]{}
		;}=0, \quad \text{that is,} \quad b_i y_i^k y_{i+1}^{\ell}b_i^*=0 \text{ for any } i \geq 1.
	\end{equation*}
\end{lemma}

\begin{proof}
	Using Relation (R4) 
	to consecutively slide dots from the right side of the loop to the left, any loop with dots as above can be written as a linear combination of loops with dots on the left only. Hence, without loss of generality, we can assume $\ell=0$.
Applying Relation (R4) and Lemma~\ref{lem:dot-slide}(c), we can rewrite a loop with $k+1$ dots on the left in two different ways (where the integers always indicate the number of dots on the strand):
	$$
	\TikZ{[scale=.5]
		\node [left] at (0,1.5) {k};
		\node [right] at (1,1.5) {};
		\draw 
		(0,1) arc(-180:0:0.5) (0,2) arc(180:0:0.5)
		(0,1) node{} to (0,2) node{} (1,1) node{} to (1,2) node{}
		(0,1.5) node[fill,circle,inner sep=1.5pt]{} (1,1.5) node[fill,circle,inner sep=1.5pt]{}
		;}+\TikZ{[scale=.5]
		\node [left] at (0,1.5) {k};
		\draw 
		(0,1) arc(-180:0:0.5) (0,2) arc(180:0:0.5)
		(0,1) node{} to (0,2) node{} (1,1) node{} to (1,2) node{}
		(0,1.5) node[fill,circle,inner sep=1.5pt]{} 
		;}=
	\TikZ{[scale=.5]
		\node [left] at (0,1.5) {k};
		\node [left] at (0,1) {};
		\draw 
		(0,1) arc(-180:0:0.5) (0,2) arc(180:0:0.5)
		(0,1) node{} to (0,2) node{} (1,1) node{} to (1,2) node{}
		(0,1.5) node[fill,circle,inner sep=1.5pt]{}  (0,1) node[fill,circle,inner sep=1.5pt]{} (0,1.5) node[V]{}
		;}
	=
	\TikZ{[scale=.5]
		\node [left] at (0,1.5) {k+1};
		\draw 
		(0,1) arc(-180:0:0.5) (0,2) arc(180:0:0.5)
		(0,1) node{} to (0,2) node{} (1,1) node{} to (1,2) node{}
		(0,1.5) node[fill,circle,inner sep=1.5pt]{} 
		;}=
	\TikZ{[scale=.5]
		\node [left] at (0,1.5) {k};
		\node [left] at (0,2) {};
		\draw 
		(0,1) arc(-180:0:0.5) (0,2) arc(180:0:0.5)
		(0,1) node{} to (0,2) node{} (1,1) node{} to (1,2) node{}
		(0,1.5) node[fill,circle,inner sep=1.5pt]{}  (0,1.5) node[fill,circle,inner sep=1.5pt]{} (0,2) node[fill,circle,inner sep=1.5pt]{}
		;}
	=\TikZ{[scale=.5]
		\node [left] at (0,1.5) {k};
		\node [right] at (1,1.5) {};
		\draw 
		(0,1) arc(-180:0:0.5) (0,2) arc(180:0:0.5)
		(0,1) node{} to (0,2) node{} (1,1) node{} to (1,2) node{}
		(0,1.5) node[fill,circle,inner sep=1.5pt]{}  (1,1.5) node[fill,circle,inner sep=1.5pt]{}
		;}-\TikZ{[scale=.5]
		\node [left] at (0,1.5) {k};
		\draw 
		(0,1) arc(-180:0:0.5) (0,2) arc(180:0:0.5)
		(0,1) node{} to (0,2) node{} (1,1) node{} to (1,2) node{}
		(0,1.5) node[fill,circle,inner sep=1.5pt]{} 
		;}.$$
	
	Subtracting $
	\TikZ{[scale=.5]
		\node [left] at (0,1.5) {k};
		\node [right] at (1,1.5) {};
		\draw 
		(0,1) arc(-180:0:0.5) (0,2) arc(180:0:0.5)
		(0,1) node{} to (0,2) node{} (1,1) node{} to (1,2) node{}
		(0,1.5) node[fill,circle,inner sep=1.5pt]{}  (1,1.5) node[fill,circle,inner sep=1.5pt]{}
		;}$ from both sides, we get $2\left(\TikZ{[scale=.5]
		\node [left] at (0,1.5) {k};
		\draw 
		(0,1) arc(-180:0:0.5) (0,2) arc(180:0:0.5)
		(0,1) node{} to (0,2) node{} (1,1) node{} to (1,2) node{}
		(0,1.5) node[fill,circle,inner sep=1.5pt]{} 
		;}\right)=0$.
\end{proof}

\begin{example} Lemma~\ref{lem:dot-loop} shows that all {\it isolated} loops, i.e. those which do not intersect any other strands, with or without dots are equal to zero. This does not mean that all dotted diagrams involving (non-isolated) loops are equal to zero, as the following example shows.
	\begin{align*}
	d=
		\TikZ{[scale=.5]
			\draw 
			(0,1) arc(-180:0:0.5) (0,4) arc(180:0:0.5)
			(0,1) node{} to (0,4) node{}
			(1,1) node{} to (2,2) node{} to (2,3) node{} to (1,4) node{}
			(2,0.5) node{} to (2,1) node{} to (1,2) node{} to (1,3) node{} to (2,4) node{} to (2,4.5) node{}			
			(1,2.5) node[fill,circle,inner sep=1.5pt]{} 
			;}&=
		\TikZ{[scale=.5]
			\draw 
			(0,1) arc(-180:0:0.5) (0,4) arc(180:0:0.5)
			(0,1) node{} to (0,4) node{}
			(1,1) node{} to (2,2) node{} to (2,3) node{} to (1,4) node{}
			(2,0.5) node{} to (2,1) node{} to (1,2) node{} to (1,3) node{} to (2,4) node{} to (2,4.5) node{}		
			(1.8,3.8) node[fill,circle,inner sep=1.5pt]{} 
			;}-
		\TikZ{[scale=.5]
			\draw 
			(0,1) arc(-180:0:0.5) (0,4) arc(180:0:0.5)
			(0,1) node{} to (0,4) node{}
			(1,1) node{} to (2,2) node{} to (2,4.5) node{}
			(2,0.5) node{} to (2,1) node{} to (1,2) node{} to (1,4) node{}		 
			;}-
		\TikZ{[scale=.5]
			\draw 
			(0,1) arc(-180:0:0.5) (0,4) arc(180:0:0.5) (1,4) arc(-180:0:0.5)
			(0,1) node{} to (0,4) node{}
			(1,1) node{} to (2,2) node{}
			(1,2) arc(180:0:0.5)
			(2,0.5) node{} to (2,1) node{} to (1,2) node{}
			(2,4) node{} to (2,4.5) node{}		
			;}=
		\TikZ{[scale=.5]
			\draw 
			(0,1) arc(-180:0:0.5) (0,2.5) arc(180:0:0.5)
			(0,1) node{} to (0,2.5) node{}
			(1,1) node{} to (1,2.5) node{}
			(2,0.5) node{} to (2, 3) node{}		
			(2,1.75) node[fill,circle,inner sep=1.5pt]{} 
			;}-
		\TikZ{[scale=.5] \draw
			(0,0.5) node{} to (0,1) node{} to (1,2) node{} to (1,2.5) node{}
			(0,2.5) arc(180:0:0.5) 
			(2,1) node{} to (2,3)node{} 
			(1,1) node{} to (0,2) node{} to (0,2.5)node{} 
			(1,1) arc(-180:0:0.5)
			;}-
		\TikZ{[scale=.5]
			\draw 
			(0,1) arc(-180:0:0.5) (0,2.5) arc(180:0:0.5) (1,2.5) arc(-180:0:0.5) (1,1) arc(180:0:0.5)
			(0,1) node{} to (0,2.5) node{}
			(2,0.5) node{} to (2,1) node{}
			(2,2.5) node{} to (2,3) node{}
			;}  \\
		&=
		\TikZ{[scale=.5]
			\draw 
			(0,1) arc(-180:0:0.5) (0,2.5) arc(180:0:0.5)
			(0,1) node{} to (0,2.5) node{}
			(1,1) node{} to (1,2.5) node{}
			(2,0.5) node{} to (2, 3) node{}		
			(2,1.75) node[fill,circle,inner sep=1.5pt]{} 
			;}-
		\TikZ{[scale=.5] \draw
			(0,0.5) node{} to (0,2.5) node{}
			(0,2.5) arc(180:0:0.5) 
			(2,1) node{} to (2,3)node{} 
			(1,1) node{} to (1,2.5) node{}
			(1,1) arc(-180:0:0.5)
			;}+
		\TikZ{[scale=.5]
			\draw 
			(0,0) node{} to (0,2.5) node{}
			;}=
		0
		+
		\TikZ{[scale=.5] \draw
			(0,0) node{} to (0,2.5) node{}
			;}+
		\TikZ{[scale=.5]
			\draw 
			(0,0) node{} to (0,2.5) node{};}=
		2 \TikZ{[scale=.5]
			\draw 
			(0,0) node{} to (0,2.5) node{};}
	\end{align*}	
	Note that although $d$ has one dot, but above calculation shows that it can be rewritten as a diagram with no dots. This is a general phenomenon - resolving loops in a diagram with $k$ dots will produce a linear combination of diagrams without loops which all have $<k$ dots (see the proof of Proposition~\ref{oVW-span}). 
\end{example}

\subsection{Spanning set}
We now prove the first part of Theorems \ref{Thm1} and \ref{Thm2} - namely, that the sets $S_{a,b}$ and $S^{\leq k}_{a,b}$ span $\Hom_{\sBr}(a,b)$ and $\Hom_{\sVW}(a,b)^{\leq k}$, respectively.

\begin{lemma}\label{choice of Sab is a choice of sign}
	If $d_1,d_2$ in $\Hom_{\sBr}(a,b)$ are any two normal diagrams with the same connector, $P(d_1)=P(d_2)$, then $d_1=d_2 \in \Hom_{sBr}(a,b)$. 
\end{lemma}
\begin{proof}
As they are both normal,  the diagrams $d_1$ and $d_2$ differ by at most the order of the crossings, so by braid relations (R1), $d_1=d_2$ in $\sBr$.
\end{proof}

\begin{proposition}\label{oBr-span} Any diagram $d$ in $\sBr$ is either equal to zero (if it has loops) or (if it has no loops) to $\pm d_c\in S_{a,b}$, where $c=P(d)$ is the connector corresponding to $d$. In particular, $S_{a,b}$ spans $\Hom_{\sBr}(a,b)$.
\end{proposition}
\begin{proof}
	If the diagram $d \in \sBr$ has any loops, we can use relations (R1) -- (R3) together with Lemma~\ref{lem:sBRrel} to isolate the loops to one side, which shows $d=0$.  
	
	If the diagram has no loops, we can use relations (R1) -- (R3) and Lemma~\ref{lem:sBRrel} to eliminate any self intersections, double intersections (two strings intersecting twice), and change the height of cups and caps. The resulting normal diagram $d'$ will have the same connector as $d$, $c:=P(d)=P(d')$,  and it will differ from $d$ in $\sBr$ by possibly a sign, $d=\pm d'$. It will possibly differ from $d_c\in S_{a,b}$ by the order of the crossings, so by Lemma~\ref{choice of Sab is a choice of sign} it satisfies $d'=d_c$. Thus, $d=\pm d_c$.
\end{proof}

The situation is only slightly more involved for $\sVW$, as transforming a diagram to an element of $S_{a,b}^\bullet$ can produce additional terms with fewer dots, in effect replacing the diagram by 
 a linear combination of elements of $S_{a,b}^\bullet$. More precisely we have

\begin{proposition}\label{oVW-span} Any dotted diagram $d \in \Hom_{\sVW}(a,b)^{\le k}$ is equal to a linear combination of elements in $S_{a,b}^{\le k}$. 
\end{proposition}
\begin{proof}
We argue by induction on $k$, with $k=0$ given by Proposition~\ref{oBr-span}. Assume $k\ge 1$, and let $d$ be a diagram with $k$ dots.

If $d$ contains loops, work with one loop at a time to: 
	\begin{enumerate}[(i)]
		\item slide all the dots on the loop so they are all to the left; 
		\item slide any dots on other strings away from the loop, so that no dots are in the interior of the loop.
	\end{enumerate}
This is accomplished using (R4) and Lemma~\ref{lem:dot-slide}. At each step, we get a linear combination of one diagram with the same number of dots, which are now in a better position, i.e.\  further away from the interior of a loop or more to the left on a loop, and diagrams with fewer dots. Applying the induction assumption to diagrams with fewer dots, it is enough to prove the claim for the diagram with all the dots on loops moved all the way to the left, and no dots in the interior of loops. For such a diagram, any loop can be moved away from the other strings, so by Lemma~\ref{lem:dot-loop} that diagram is equal to zero. This proves the claim for dotted diagrams with loops.

Next, assume that $d$ has no loops. Working with one string at a time,
	\begin{enumerate}[(i)]
		\item slide all the dots on through strings to the bottom.
		\item slide the dots on cups and caps all the way to the left.
	\end{enumerate}
Again, this is done using (R4) and Lemma~\ref{lem:dot-slide}. At the end of this process, we have replaced $d$ by a linear combination of a diagram $d'$ with $k$ dots (which are all the way on the bottom of through strings, and on the left of cups and caps), plus diagrams with fewer dots. Apply the induction assumption to diagrams with fewer dots; it remains to prove the claim for $d'$. The position of dots on $d'$ means that it is of the form $\prod_i  y^{a_i}_i d'' \prod_j y^{b_j}_j$ for some $a_i,b_j \in \NN_0$ and some undotted diagram $d''\in \sBr$. 
Applying Proposition~\ref{oBr-span} to $d''$ completes the proof. 
\end{proof}

\subsection{A flipping functor $\iota: \sVW \rightarrow \sVW^{\text{op}}$}

We describe a functor between the supercategory $\sVW$ and its opposite, which on the level of diagrams corresponds to an upside-down flip, with some additional signs.

\begin{proposition}\label{invol} There is an isomorphism of supercategories $\iota: \sVW \rightarrow \sVW^{\text{op}}$, given on objects by the identity and on morphisms by:
	$$\iota(s_i)=-s_i, \quad	\iota(b_i)=b_i^*, \quad \iota(b_i^*)=-b_i,  \quad	\iota(y_i)=-y_i.$$
The inverse functor is given by $\iota^3$. It restricts to an anti-isomorphism on each $\op{End}_{\sVW}(a)$, $a\in\mathbb{N}$ (sending $s_i$ , $e_i$, $y_i$ to minus themelves in the notation from  Section~\ref{secAbar}).
\end{proposition}
\begin{proof}
To see that $\iota$ respects
	the defining relations of $\sVW$, we note that (R1) and the first part of (R4) are invariant under 
	the diagrams upside-down, the flips of (R3) and the second part of (R4) are a consequence of Lemmas \ref{lem:sBRrel} and \ref{lem:dot-slide}, and the first diagram of (R2) turns into the second after the flip, with the sign changes being consistent as well. 
\end{proof}

\section{The periplectic Lie superalgebra $\fp(n)$}\label{sec:periplectic}

We recall some facts from the representation theory of the Lie superalgebra $\fp(n)$. For more details on Lie superalgebras see for instance \cite{Mu}, \cite{Se}, and for $\fp(n)$ see also \cite{BDEHHILNSS}.

\subsection{Definition and bases}
From now on, let $V=\mathbb{C}^{n|n}$ be the superspace of superdimension $n|n$, meaning $V=V_{\overline 0}\oplus V_{\overline 1}$ with $V_{\overline 0}=\C^n$, $V_{\overline 1}=\C^n$. Let $v_1, \ldots , v_n$ be the standard basis of $V_{\overline 0}$ and $v_{1'}, \ldots , v_{n'}$ be the standard basis of $V_{\overline 1}$. We let $[n] := \{1, \dots, n\}$,  $[n']:=\{1',\dots,n'\}$ denote the sets of indices.

The \emph{general linear Lie superalgebra} $\fgl(n|n)$ is the Lie superalgebra of endomorphisms of $V$,  with $\Z/2\Z-$grading induced by $V$, and the Lie superbracket given by the super commutator $[x,y]=xy-(-1)^{\overline{x}\overline{y}}yx$. 
In terms of matrices, 
$$\fgl(n|n) = \left\{  \begin{pmatrix} A & B \\ C & D\end{pmatrix} \mid 
A, B, C, D \in M_{n,n}(\CC)\right\},$$
with 
$$\fgl(n|n)_{\overline{0}} = \left\{ \begin{pmatrix} A & 0 \\ 0 & D \end{pmatrix} \right\}=\fgl(n) \oplus \fgl(n) \quad \text{and}\quad
\fgl(n|n)_{\overline{1}} = \left\{ \begin{pmatrix} 0 & B \\ C & 0 \end{pmatrix} \right\}.$$ 
We call $V$ the \emph{vector representation} of $\fgl(n|n)$. 
A basis of $\fgl(n|n)$ is given by the matrix units $E_{rs}$ for $r,s \in [n] \cup [n']$, which act on
$V$ as  
$E_{rs}v_t = \delta_{st} v_r$ for $t \in [n] \cup [n']$. 

Let $\beta: V \otimes V \to \CC$ be the bilinear form given by 
$$\beta|_{V_{\overline{0}}\otimes V_{\overline{0}}}=\beta|_{V_{\overline{1}}\otimes V_{\overline{1}}}=0 \quad \text{ and } \quad 	
\beta(v_i,v_{j'})=\beta(v_{j'},v_i)=\delta_{i,j}\hbox{ for all } i,j \in [n].$$
It is symmetric, odd, and non-degenerate on $V$. 
Andr\'{e} Weil named such forms \textit{periplectic} by analogy with symplectic forms. The corresponding {\it periplectic Lie superalgebra} $\fp(n)$  is then defined as the Lie supersubalgebra of $\fgl(n|n)$ preserving $\beta$, i.e.\ it is spanned by all homogeneous elements $x$ which satisfy $\beta(x u, v) + (-1)^{\bar{x} \bar{u}} \beta(u, xv) = 0$.
In terms of matrices,
$$\fp(n) = \left\{ \begin{pmatrix} A & B \\ C & -A^t \end{pmatrix} \in \fgl(n|n) 
\mid B = B^t, C = -C^t \right\},$$
with 
$$\fp(n)_{\overline 0} = \left\{ \begin{pmatrix} A & 0 \\ 0 & -A^t \end{pmatrix} \right\}, \quad \fp(n)_{\overline 1} = \left\{ \begin{pmatrix} 0 & B \\ C & 0 \end{pmatrix} \right\}.$$

\begin{lemma}
	The set $\cX = \{A_{ij}^- \mid i, j \in [n] \} \cup \{B_{ij}^+ \mid i \leq j \in [n]\} \cup \{C_{ij}^- \mid i < j \in [n]\}$ is a basis for $\fp(n)$, where
$A_{ij}^\pm = E_{ij} \pm E_{j' i'}$, 
$B_{ij}^\pm = E_{ij'} \pm E_{j i'}$, $C_{ij}^\pm = E_{i'j} \pm E_{j' i}$, and $\overline{A_{ij}^{\pm}}=0, \; \overline{B_{ij}^{\pm}}=\overline{C_{ij}^{\pm}}=1$.
\end{lemma}

The universal enveloping superalgebra of a Lie superalgebra $\mathfrak{g}$ is the quotient of the tensor algebra $T(\mathfrak{g})$ by the ideal generated by elements of the form $x\otimes y-(-1)^{\overline{x}\overline{y}}y\otimes x - [x,y]$ for all homogeneous $x,y\in \mathfrak{g}$.
Letting
$$\mathfrak{g}=\mathfrak{p}(n), \quad \mathfrak{g}_{-1}=\left\{ \begin{pmatrix} 0 & 0 \\ C & 0 \end{pmatrix} \in \mathfrak{p}(n) \right\},\quad \mathfrak{g}_{0}=\left\{ \begin{pmatrix} A & 0 \\ 0 & -A^t \end{pmatrix} \in \mathfrak{p}(n) \right\}, \quad \mathfrak{g}_{1}=\left\{ \begin{pmatrix} 0 & B \\ 0 & 0 \end{pmatrix} \in \mathfrak{p}(n) \right\},$$
the PBW-Theorem for $\mathfrak{p}(n)$ theorem states that multiplication gives an isomorphism of vector superspaces
$$\Lambda (\mathfrak{g}_{1}) \otimes S(\mathfrak{g}_{0}) \otimes \Lambda (\mathfrak{g}_{-1}) \to \mathcal{U}(\fp(n)).$$

There is a {\it supertrace form} on $\fgl(n|n)$, given by 
\begin{equation}\label{eqn:str}\<x,y\> = \str(xy), \quad  \text{ with } \quad \str\left(\begin{pmatrix} A&B\\C&D \end{pmatrix}\right) = \tr(A) - \tr(D).\end{equation}
It is bilinear, invariant in the sense $\<[x,y],z\>=\<x,[y,z]\>$ for all $x,y,z \in \fgl(n|n)$, and nongdegenerate. The subalgebra $\fp(n)$ is isotropic with respect to this form; however, one can consider the dual space $\fp(n)^\perp$ of $\fp(n)$ in $\fgl(n|n)$ with respect to this form, which satisfies $\fgl(n|n) = \fp(n) \oplus \fp(n)^\perp$. The basis $\cX$ of $\fp(n)$ gives rise to a dual basis $\cX^*=\{x^* \mid x \in \cX\}$ for $\fp(n)^{\perp}$, in the sense that $\<x^*,y\>=\delta_{xy} \; \; \forall \; y \in \cX$. It is explicitly given as $$(A^-_{ij})^* = \half A_{ji}^+, \quad (B_{ij}^+)^* = -\half C_{ji}^+, \quad (B_{ii}^+)^* = -\frac{1}{4} C_{ii}^+, \quad  \text{ and } \quad (C_{ij}^-)^* = \half B_{ji}^-.$$

\subsection{The category $\pn$}
We consider the monoidal supercategory $\pn$ of representations of $\fp(n)$ with the set $\Hom_{\fp(n)}(M,N)$ of morphisms from $M$ to $N$ given by linear combinations of homogeneous $\C$-linear maps $f$ from $M$ to $N$ such that $f(x.m)=(-1)^{\overline{x}\overline{f}}x.f(m)$ for homogeneous elements $m\in M$, $x\in \fp(n)$. We in particular allow morphisms to be odd (i.e.\  they change the parity of elements they are applied to).

This supercategory is symmetric, with the braiding given by the superswap
$$\sigma: M\otimes N\to N\otimes M, \qquad \sigma(m\otimes n)=(-1)^{\overline{m}\overline{n}}n\otimes m.$$

We call $V$ the vector representation of $\mathfrak{p}(n)$. The form $\beta$ induces an (odd) identification of $V\to V^*$ as $\mathfrak{p}(n)$-representations, given by 
$v\mapsto \beta(v,-).$ Similarly, the bilinear form $(\beta\otimes \beta)\circ (1\otimes \sigma\otimes 1):V^{\otimes 4}\to \C$ induces an identification $(V\otimes V)^*\to V\otimes V$. 
 With that, the dual map to the form $\beta$ can be thought of as $\beta^*:\mathbb{C}\to V\otimes V$; it is given by
$$\beta^*(1)=\sum_{i}(v_i \otimes v_{i'} - v_{i'}\otimes v_i).$$

\begin{lemma}
	\label{lem:homsbb*s}
	The following are maps of Lie superalgebra modules of degrees $\overline{1}, \; \overline{1},$ and $\overline{0}$:
	$$\beta  \in \Hom_{\fp(n)}(V\otimes V,\C), \quad
	\beta^* \in \Hom_{\fp(n)}(\C,V\otimes V), \quad
	\s \in \Hom_{\fp(n)}(V\otimes V,V \otimes V).$$
\end{lemma}

\subsection{A (fake) quadratic Casimir element}\label{Omega}

Because of the absence of the Killing form on $\mathfrak{p}(n)$, there is no Casimir element in $\mathcal{U}(\mathfrak{p}(n))$, nor a quadratic Casimir in $\mathfrak{p}(n)\otimes \mathfrak{p}(n)$. (In fact, the centre of $\mathcal{U}(\mathfrak{p}(n))$ is trivial.) We can however use the supertrace form on $\mathfrak{gl}(n|n) $ to define a \emph{fake Casimir} in  $\mathfrak{p}(n)\otimes \mathfrak{gl}(n|n)$ as follows (see also \cite{BDEHHILNSS}). Let 

$$\Omega = 2\sum_{x \in \cX} x \otimes x^* \in \fp(n) \otimes \fgl(n|n);$$
explicitly,
\begin{equation}
\label{eq:Casimir}
\Omega = 
\sum_{i,j} A_{ij}^- \otimes A_{ji}^+ 
- \half \sum_{i} B_{ii}^+ \otimes C_{ii}^+
- \sum_{i < j} B_{ij}^+ \otimes C_{ji}^+
+ \sum_{i < j} C_{ij}^- \otimes B_{ji}^-.
\end{equation}

This element does not act on an arbitrary tensor product $M\otimes N$ of $\fp(n)$-representations, but acts on $M \otimes V$, for $M$ any $\fp(n)$-representation, and $V$ the above described vector representation.   
Its action gives a morphism in $\pn$ by the following proposition, first observed in \cite[Lemma 4.1.4]{BDEHHILNSS}.

\begin{proposition}\label{prop:Omega-fpn-commute}
	The actions of $\Omega$ and $\fp(n)$ on $M \otimes V$ commute, i.e.\ 
	$\Omega\in\operatorname{End}_{\fp(n)}(M\otimes V)$.	
\end{proposition}
\begin{proof} The Lie superalgebra $\fp(n)$ acts on $M \otimes V$ via the coproduct $\Delta$ of $\mathcal{U}(\fp(n))$, given by $\Delta(y)=y \otimes 1 + 1 \otimes y$. For any homogeneous element $y \in \fp(n) \subset \fgl(n|n)$, we have
	\begin{equation*}[y \otimes 1 + 1 \otimes y,x_i\otimes x^*_i]=[y,x_i]\otimes x^*_i + (-1)^{\bar{y}\bar{x_i}}x_i \otimes [y,x^*_i].\end{equation*}	
Furthermore, by expanding in the basis $\{x_i\}_i\cup\{x^*_i\}_i$ of $\fgl(n|n)$, we can see that	
	$$[y, x_i] = \sum_j \<x_j^*,[y, x_i]\> x_j, \quad  \text{ and } \quad 
	{[y, x_i^*] = \sum_j \<[y, x_i^*], x_j\> x_j^*},$$
Therefore, using the invariance of the supertrace form (\ref{eqn:str}),
	\begin{align*}
	[\Delta(y), \Omega] &= [ y \otimes 1 + 1 \otimes y, \sum_{i} x_i \otimes x_i^*] = \sum_i [y, x_i] \otimes x_i^* +  \sum_i (-1)^{\bar{y} \bar{x}_i} x_i \otimes [y, x_i^*]\\
	&= \sum_{i,j} \<x_j^*,[y, x_i] \> (x_j \otimes x_i^*) + \sum_{i,j}(-1)^{\bar{y} \bar{x}_i}  \<[y, x_i^*],x_j \> (x_i \otimes x_j^*) \\
	&= \sum_{i,j} \<x_j^*,[y, x_i] \> (x_j \otimes x_i^*) - \sum_{i,j} \<[x_i^*,y], x_j\> (x_i \otimes x_j^*) \\
	&= \sum_{i,j} \<x_j^*,[y, x_i] \> (x_j \otimes x_i^*) - \sum_{i,j} \<x_i^*, [y,x_j]\> (x_i \otimes x_j^*)  = 0.\hfill\qedhere
	\end{align*}
\end{proof}

\remark \label{rem:Omegaparity} Note that $\Omega$ is even, $\overline{\Omega}=\overline{0}$, since from (\ref{eq:Casimir}) we see that
$$\Omega \in \left(\fgl(n|n)_{\bar{1}} \otimes \fgl(n|n)_{\bar{1}}\right) \oplus \left(\fgl(n|n)_{\bar{0}} \otimes \fgl(n|n)_{\bar{0}}\right) \subset (\fgl(n|n) \otimes \fgl(n|n))_{\bar{0}}.$$

\vspace{0.3cm}

We consider the special case when $M=V$, and calculate the action of $\Omega$ in that case.

\begin{lemma}\label{lem:OmegaOnVV}
	The action of $\Omega$ on $V \otimes V$ is explicitly given by $\s + \beta^* \beta$. 
\end{lemma}

\begin{proof}
	
This is an explicit calculation in the basis $\{v_a \otimes v_b \mid a,b \in [n] \cup [n'] \}$ of $V \otimes V$. We include the computation for the case $a,b\in [n]$. The remaining three cases follow similarly.

Let $a,b\in [n]$. Then 
	\begin{align*}
	(A_{ij}^- \otimes A_{ji}^+)(v_a \otimes v_b) 
	& = A_{ij}^- v_a \otimes A_{ji}^+ v_b 
	= \delta_{aj}v_i \otimes \delta_{bi}v_j
	= \delta_{aj}\delta_{bi} (v_b \otimes v_a),\\
	(B_{ij}^+ \otimes C_{ji}^+) (v_a \otimes v_b)  
	&= B_{ij}^+ v_a  \otimes C_{ji}^+ v_b = 0, \text{ and }\\
	(C_{ij}^- \otimes B_{ji}^-) (v_a \otimes v_b)  
	&= C_{ij}^- v_a  \otimes B_{ji}^- v_b = 0,
	\end{align*}
	and therefore
	$\Omega (v_a \otimes v_b) 
	= \sum_{i, j} \delta_{aj}\delta_{bi} v_b \otimes v_a + 0+ 0 + 0 = v_b \otimes v_a=(\s+\beta^*\beta)(v_a \otimes v_b).$
\end{proof}

\subsection{Jucys-Murphy type elements}\label{JM}
Once we have the above fake Casimir operator, we can define certain commuting elements of $\End_{\fp(n)}(M \otimes V^{\otimes a})$. They are intended to mimic the action of the polynomial generators of the degenerate affine Hecke algebra in case of $\mathfrak{gl}(n)$.

Label the tensor factors of $M\otimes V^{\otimes a}$ by $0, 1,\ldots, a$, and let $\Omega_{ij}$ denote the operator acting as $\Omega$ applied to the $i$th and $j$th factor and the identity everywhere else. For $1\le j \le a$, let 
$$Y_j = \sum_{i=0}^{j-1} \Omega_{i j} \; \; \in \; \; \End_{\fp(n)}(M\otimes V^{\otimes a}),$$
(see \cite[Section 4.1]{BDEHHILNSS}) The following result is then standard. 
\begin{proposition} The operators $Y_1, Y_2, \dots, Y_a$ pairwise commute.  
\end{proposition}
\begin{proof}
Now $\Omega$ commutes with the coproduct $\Delta(y)$, $y \in \fp(n)$, so $\Omega \otimes 1 =Y_1$  commutes with 
	$$(\Delta \otimes 1)\Omega = \sum_{x \in \cX} \Delta(x) \otimes x^* = \sum_{x \in \cX}(x \otimes 1 \otimes x^* + 1 \otimes x \otimes x^*)=Y_2.$$
	As operators on $M \otimes V \otimes V$, this says that $Y_{1}$ commutes with $Y_2$. 
	Using $\Delta^j$ to denote the iterated coproduct $\mathfrak{p}(n)\to \mathfrak{p}(n)^{\otimes j}$, by induction we get that
	$$Y_j = (\Delta^j \otimes 1) \Omega \quad \text{ commutes with } \quad Y_k  =  \sum_{i=0}^{k-1} \Omega_{i, k} \text{ for $k<j$},$$
	since $\Delta^j(x)$ for $x\in \cX$ commutes with $\Omega_{i, k}$ for $i, k < j$. \end{proof}

\remark \label{rem:rel}
There is a quotient map $\sVW \to \sBr$, determined by $y_1\to 0$, $b_i\mapsto b_i$, $b_i^*\mapsto b_i^*$, $s_i\mapsto s_i$. Under this quotient map, 
$$y_j\mapsto \sum^{j-1}_{i=1} \left(\hspace{-0.15cm}
\TikZ{[scale=.5]
	\draw 
	(0,0) node{} to (1,1.5) node{}
	(1,0) node{} to (0,1.5) node{}
	(1.5,0.75) node{+}
	(2,0) arc(180:0:0.5) (2,1.5) arc(-180:0:0.5) (2,2) node{};}\right)_{ij}.
$$
These commuting elements of $\sBr_a$ are the analogues of Jucys-Murphy elements for the symmetric group or the Brauer algebra, see \cite{CST} and \cite[Section 2]{Nazarov}. As elements of the superalgebra $\sBr_a$, they were independently defined in \cite[Section 6]{C}, and their eigenvalues are then used, following the approach of \cite{OV}, to study the representation theory of $\sBr_a$ and consequently $\fp(n)$. In terms of the action on $M\otimes V^{\otimes a}$, taking the cyclotomic quotient determined by $y_1\mapsto 0$ corresponds to taking $M$ to be the trivial module (see Lemma~\ref{lem:OmegaOnVV}). This recovers the action of $\sBr$ on $V^{\otimes a}$ from \cite{M}. 

\remark \label{rem:rel2}
We have the following relation in $\mathrm{Hom}_{\fp(n)}(M\otimes V^{\otimes a})$, for any $1\leq j<a$, which can be checked directly: 
$$\Omega_{i,j+1}=\s_j\;\Omega_{ij}\;\s_j \text{ for } i<j.$$

\subsection{The functor $\Psi^M_n$}	\label{sec:funct}

The diagrammatically described supercategory $\sVW$ can be related to $\pn$ and used to study the representation theory of the periplectic Lie superalgebra.

Analogous to the notation $\Omega_{ji}$, we will denote by $\s_i, \beta_i$ and $\beta^*_i$ the operators acting as $\s, \beta$ and $\beta^*$ in the $i$th and $(i+1)$st positions of a tensor product $M \otimes V^{\otimes a}$, and identity elsewhere. Here, $M$ is considered as the $0$th factor.

\begin{proposition}
	\label{prop:funct}
	For any $M \in \pn$, there is a superfunctor $\Psi^M_n: \sVW \longrightarrow \pn$ defined on objects by $a \mapsto M \otimes V^{\otimes a}$ and on morphisms by
	$$
	s_i \mapsto  \sigma_i, \quad\quad
	b_i \mapsto  \beta_i, \quad\quad
	b_i^* \mapsto \beta^*_i, \quad\quad
	y_i \mapsto Y_i=\sum_{0\le j <i} \Omega_{ji}.
	$$
\end{proposition}


\begin{proof}[Proof of Proposition~\ref{prop:funct}]
	From Lemma~\ref{lem:homsbb*s} and Proposition~\ref{prop:Omega-fpn-commute}, we know that $\beta, \beta^*,\s$, and $\Omega$ are morphisms in $\pn$, hence so are the images of $s_i, b_i, b^*_i, y_i$ under $\Psi_n^M$. Furthermore, $\Psi^M_n$ preserves parity, since $\overline{s_i}=\overline{\s_i}=0$, $\overline{b_i}=\overline{\beta_i}=\overline{b^*_i}=\overline{\beta^*_i}=1$, and $\overline{y_i}=\overline{\sum_{0\le j <i} \Omega_{ji}}=0$, see Remark \ref{rem:Omegaparity}. It remains to check that the images of the generating morphisms satisfy the defining relations of $\sVW$. In the calculations we suppress the $0$-th tensor factor $M$.
	
	\begin{enumerate}[(R1)]
		\item \begin{enumerate}
		\item $\s_i^2=1$.
		This follows from
		$\s^2(v\otimes w)= (-1)^{\overline{v}\overline{w}}\s(w\otimes v)=(-1)^{2\overline{v}\overline{w}}v\otimes w=v\otimes w$. 
		
		\item $ \s_i\s_{i+1}\s_i= \s_{i+1}\s_i\s_{i+1}$.
		It is enough to prove this for $i=1, \; a=3$:
		\begin{align*}
		(\s_{1}\s_2\s_{1})(u\otimes v\otimes w) &=(-1)^{\overline{u}\overline{v}} (\s_{1}\s_2)(v\otimes u\otimes w)=(-1)^{\overline{u}\overline{v}+\overline{u}\overline{w}} \s_{1}(v\otimes w\otimes u) \\
		&=(-1)^{\overline{u}\overline{v}+\overline{u}\overline{w}+\overline{v}\overline{w}} w\otimes v\otimes u=(\s_{2}\s_1\s_{2})(u\otimes v\otimes w).
		\end{align*}
		\end{enumerate}
		
		\item \begin{enumerate}
		\item $\beta_i \beta_{i+1}^*=-1$.
		It is enough to prove this for $i=1$:
		\begin{align*}
		\beta_{1} \beta_{2}^*(v)&=(-1)^{\overline v}\beta_{1} (v \otimes \beta^*(1))
		=(-1)^{\overline v}\beta_{1} \left(v \otimes 
		\left( \sum_{i=1}^n v_i\otimes v_{i'} -v_{i'}\otimes v_i
		\right)
		\right) \\
		&=(-1)^{\overline v}\sum_{i=1}^ n (\beta(v,v_i) v_{i'} -\beta(v,v_{i'})v_i)=-v.
		\end{align*}
		The last equality is easily checked on every $v=v_j$, $j\in [n]\cup[n']$.
		
		\item $\beta_{i+1} \beta_{i}^*=1$.
		Similar. 
		\end{enumerate}
		
		\item \begin{enumerate}
		\item $\s_{i+1}\beta_{i}^*=\s_{i}\beta_{i+1}^*$.
		It is enough to prove this for $i=1$:
		\begin{align*}
		\s_{2}\beta_{1}^*(v) &= \s_2\left(\sum_{i=1}^n(v_i\otimes v_{i'} -v_{i'}\otimes v_i )\otimes v\right)= \sum_{i=1}^n ((-1)^{\overline{v}}v_i\otimes v\otimes v_{i'} - v_{i'}\otimes v\otimes v_i),\\
		\s_{1}\beta_{2}^*(v)&=\sum_{i=1}^n ((-1)^{\overline v}v_i\otimes v \otimes v_{i'} - (-1)^{\overline v+\overline v}v_{i'}\otimes v \otimes v_i).
		\end{align*}
		
		\item 
		$\s_{i}\beta_{i}^*=-\beta_{i}^*$.
		This follows from the fact that $\beta^*(1)$ is skew supersymmetric. Note that this, together with the previous relations, also implies that $\beta_i\s_i=\beta_i$ and $\beta^*_i\beta_i=0$, which will be used in proving (R4)(b).
		\end{enumerate}
		\item  \begin{enumerate}
			\item \label{reln-y2=second}
		$Y_{i+1}=\s_iY_i\s_i+\s_i+\beta_i^*\beta_i $.
		This formula follows via the following computation, using Remarks \ref{rem:rel} and \ref{rem:rel2}, and Lemma~\ref{lem:OmegaOnVV}
		\begin{align*}
			Y_{i+1}&=\sum_{0\le k <i+1} \Omega_{k,i+1}
		=\sum_{0\le k <i} \Omega_{k,i+1}+ \Omega_{i,i+1}
		=\sum_{0\le k <i} \s_i\Omega_{k,i}\s_i+ \Omega_{i,i+1} \\
		&= \s_i\left(\sum_{0\le k <i} \Omega_{k,i}\right) \s_i+\s_i+\beta_i^*\beta_i =\s_iY_i \s_i+\s_i+\beta_i^*\beta_i
		\end{align*}
		
		\item \label{beta(y1-y2)-second}
		$\beta_1(Y_1-Y_2)=-\beta_1$.
		We have $\beta \circ (x^* \otimes 1 - 1\otimes x^*)=0$ for any $x^* \in \fp(n)^{\perp}$, which can be checked directly on a basis of $V \otimes V$, and hence $\beta_1 \circ (\Omega_{01}-\Omega_{02})=0$. It follows that 
		$\beta_1(\Omega_{01}-\Omega_{02}-\Omega_{12})=-\beta_1\Omega_{12}=-\beta_1 (\s_1+\beta_1^*\beta_1)=-\beta_1\s_1+0=-\beta_1$.\qedhere
	\end{enumerate}	
	\end{enumerate}
\end{proof}

\section{Linear independence of $S_{a,b}^\bullet$} 
\label{linindep}
The purpose of this section is to prove linear independence of the sets $S_{a,b}$ and $S_{a,b}^\bullet$, and thus prove Theorems \ref{Thm1} and \ref{Thm2}. The idea is to exploit a close connection of $\sVW$ and the representation theory of the periplectic Lie superalgebra $\mathfrak{p}(n)$. Namely, as explained in Proposition~\ref{prop:funct}, for every $n$ and every $\mathfrak{p}(n)$-representation $M$, the functor $\Psi_n^M:\sVW \to \pn$ gives a way of interpreting diagrams $d\in \Hom_{\sVW}(a,b)$ as linear $\mathfrak{p}(n)$-homomorphisms $\Psi_n^M(d):M\otimes V^{\otimes a}\to M\otimes V^{\otimes b}$. For given $a,b$, and $k$ in $\mathbb{N}_0$, we will pick $n$ and an appropriate $M\in \pn$ so that the corresponding functor $\Psi_n=\Psi_n^M:\sVW \to \pn$ maps $S_{a,b}^{\le k}$ to a linearly independent set in $\Hom_{\fp(n)}(M\otimes V^{\otimes a}, M\otimes V^{\otimes b})$.

The argument for linear independence is slightly easier in the associated graded setting. For that purpose, we define an auxillary category $\GVW$ and auxilary functors $\Phi_n$, which will turn out to be the associated graded of $\sVW$ and $\Psi_n$. This is analogous to the structure of the main proof in \cite{BCNR}, where a close connection between the affine oriented Brauer category and $\mathcal{W}$-algebras is exploited to construct certain functors, which are then used to prove linear independence. We start with some preliminaries about filtrations and gradings.

\subsection{Graded and filtered supercategories}

An $\NN_0$-filtered superspace is a superspace $U$ with a filtration by subspaces $\{0\}=U^{\leq -1} \subseteq U^{\leq 0} \subseteq U^{\leq 1} \subseteq \hdots \subseteq U$, and $U=\bigcup_{k \geq 0} U^{\leq k}$. 
A supercategory $\mathcal{C}$ such that for every $M,N \in \mathcal{C}$, $\Hom_{\mathcal{C}}(M,N)$ has a fixed filtration compatible with composition of morphisms, $\Hom_{\mathcal{C}}(M,N)^{\leq k} \times \Hom_{\mathcal{C}}(N,P)^{\leq \ell} \rightarrow \Hom_{\mathcal{C}}(M,P)^{\leq (k+\ell)}$ is a supercategory $\mathcal{C}$ enriched in the category of filtered superspaces (that is in the category whose objects are filtered superspaces and morphisms are homogeneous linear maps of degree zero). We call such a supercategory a \emph{filtered supercategory}. A \emph{graded supercategory} is a supercategory enriched in graded superspaces; this means its morphism spaces are graded superspaces, and composition is a homogeneous linear map of degree zero.

We say a functor $F:\mathcal{C} \to \mathcal{D}$ between two filtered (respectively, graded) supercategories $\mathcal{C}$ and $\mathcal{D}$ is \emph{filtered} (respectively, \emph{graded}) if it preserves the filtration (respectively, grading) on the morphism spaces. 

Now assume we have a filtered supercategory $\mathcal{C}$. Its \emph{associated graded supercategory} $gr\mathcal{C}$ is the graded supercategory with the same objects as $\mathcal{C}$, and morphism spaces the graded superspaces $\Hom_{gr\mathcal{C}}(M,N)=gr (\Hom_{\mathcal{C}}(M,N))= \bigoplus_{k\geq 0}\Hom_{gr\mathcal{C}}(M,N)^k$, where $\Hom_{gr\mathcal{C}}(M,N)^k=\Hom_{\mathcal{C}}(M,N)^{\leq k}/\Hom_{\mathcal{C}}(M,N)^{\leq (k-1)}$. 

A filtered functor $F:\mathcal{C} \to \mathcal{D}$ between two filtered supercategories induces a graded functor $gr(F):gr\mathcal{C} \to gr\mathcal{D}$. The functor $gr(F)$ is equal to $F$ on objects, and takes the associated graded map of $F$ on the morphism superspaces.

\subsection{The supercategories $\FVect$ and $\GVect$, and the functor $G$}

Let $\FVect$ be the supercategory with objects $\NN_0$-filtered superspaces, and morphisms given by the filtered superspaces
$\Hom_{\FVect}(M,N)=\bigcup_{k \in \NN_0}\Hom_{\FVect}(M,N)^{\leq k}$, where $\Hom_{\FVect}(M,N)^{\leq k}=\{f:M\rightarrow N \mid f \textrm{ linear, } f(M^{\leq i}) \subseteq N^{\leq (i+k)} \textrm{ for all } i\}$. This is an $\NN_0$-filtered supercategory as above.

Similarly, let $\GVect$ denote the supercategory whose objects are $\NN_0$-graded superspaces, and whose morphisms are superspaces of linear maps equipped with the grading coming from the objects, that is $\Hom_{\GVect}(M,N)=\bigoplus_{k \in \NN_0}\Hom_{\GVect}(M,N)^{k}$, where $\Hom_{\GVect}(M,N)^{k}=\{f:M\rightarrow N \mid f \textrm{ linear, } f(M^{i}) \subseteq N^{(i+k)} \textrm{ for all } i\}$. It is an $\NN_0$-graded supercategory in the above sense.

In particular, we can consider the associated graded category $gr(\FVect)$ described above. (Note that $gr(\FVect)$ and $\GVect$ are not the same categories; objects of $gr(\FVect)$ are filtered while objects of $\GVect$ are graded vector superspaces.)

There is a functor $G:gr(\FVect)\to \GVect$ which associates to a filtered superspace $M=\bigcup_i M^{\leq i}$ its associated graded superspace $G(M)=gr(M)=\bigoplus_i M^{\leq i}/M^{\leq (i-1)}$. On morphisms $G:\Hom_{\FVect}(M,N)^{\leq k}/\Hom_{\FVect}(M,N)^{\leq (k-1)} \to \Hom_{\GVect}(gr(M), gr(N))^k$ is given on $f\in \Hom_{\FVect}(M,N)^{\leq k}$ and $m\in M^{\le i}$
by $$G(f+\Hom_{\FVect}(M,N)^{\leq (k-1)})(m+M^{\le (i-1)})=f(m)+N^{\leq (k+i-1)}.$$

\subsection{$\sVW$ as a filtered supercategory}

The affine VW supercategory $\sVW$ can be viewed as a filtered supercategory, with the filtration on the morphism spaces given by the number of dots. Let $gr(\sVW)$ be its associated graded supercategory, defined as above. In particular, the following relations hold in $gr(\sVW)$: 
\begin{equation*} \tag{grR-4}
\begin{aligned}
\TikZ{[scale=.5] \draw
(0,0) node{} to (0,2) node{} 
(1,0) node{} to (1,2) node{} (1,1) node[fill,circle,inner sep=1.5pt]{} 
;}&  =
\TikZ{[scale=.5] \draw
(0,0) node{} to (1,1) node{}  to (0,2) node{} 
(1,0) node{} to (0,1) node{}  to (1,2) node{} (0,1) node[fill,circle,inner sep=1.5pt]{} 
;} & \in \Hom_{gr(\sVW)}(2,2)^1=\Hom_{\sVW}(2,2)^{\le 1}/ \Hom_{\sVW}(2,2)^{\le 0},\\
\TikZ{[scale=.5] \draw 
(0,1) arc(180:0:0.5) (0,0) node{} to (0,1) node{} (1,0) node{} to (1,1) node{} (1,0.5) node[fill,circle,inner sep=1.5pt]{} 
;}&=
\TikZ{[scale=.5] \draw 
(0,1) arc(180:0:0.5) (0,0) node{} to (0,1) node{} (1,0) node{} to (1,1) node{} (0,0.5) node[fill,circle,inner sep=1.5pt]{} 
;} &\in \Hom_{gr(\sVW)}(2,0)^1=\Hom_{\sVW}(2,0)^{\le 1}/ \Hom_{\sVW}(2,0)^{\le 0}. \nonumber
\end{aligned}
\end{equation*}

It is however not a priori obvious that these, along with (R1)-(R3), are the only defining relations for $gr(\sVW)$. In general, given a filtered algebra or a category, describing its associated graded by generators and relations is a nontrivial problem, and the solution to this problem usually goes most of the way towards proving a basis theorem for the filtered version (as basis theorems for graded versions are usually easier). With that in mind, we define another category $\GVW$ by generators and relations, and prove in Section \ref{final-proof-2} that $gr(\sVW)$ and $\GVW$ are indeed isomorphic as graded supercategories.

\subsection{The category $\GVW$}
Let $\GVW$ be the $\mathbb{C}$-linear monoidal supercategory generated as a monoidal supercategory by a single object $\obj$, morphisms $s=\TikZ{[scale=.5] \draw (0,0) node{} to (1,1)node{} (1,0) node{} to (0,1)node{} ;}:\;\obj\otimes \obj\longrightarrow \obj\otimes \obj$, $b=\TikZ{[scale=.5] \draw (1,0) node{} (0,0) node{} (0,0) arc(180:0:0.5) ;}:\;\obj\otimes \obj\longrightarrow\mathbbm{1}$, $b^*=\TikZ{[scale=.5] \draw (1,1) node{} (0,1) node{} (0,1) arc(-180:0:0.5) ;}:\;\mathbbm{1}\longrightarrow \obj\otimes \obj$ and $
y=\TikZ{[scale=.5] \draw (0,0) node{} to (0,1) node{} (0,0.5) node[fill,circle,inner sep=1.5pt]{};}:\;\obj \longrightarrow \obj,$ subject to relations (R1)--(R3) and (grR-4). The $\Z/2\Z$ parity is given by $\overline {s}=\overline{y}=0$, $\overline{b}=\overline{b^*}=1$. The $\mathbb{N}_0$-grading is given by $\deg s=\deg b=\deg b^*=0, \deg y=1$. Note that the imposed relations are $\NN_0$-homogeneous and so the category is well-defined. In other words, the objects of $\GVW$ are nonnegative integers, the morphisms are linear combinations of dotted diagrams, and the $\NN_0$-grading is given by the number of dots on the diagram.  

The following is analogous to Proposition~\ref{oVW-span}, and proved in exactly the same way. 
\begin{lemma}\label{span-GVW}
For any $a,b,k\in \mathbb{N}_0$, the set $S_{a,b}^k$ is a spanning set for $\Hom_{\GVW}(a,b)^k$.
\end{lemma}

\subsection{The functor $\Theta: \GVW \to gr(\sVW)$}\label{Theta}
The tautological assignments $\Theta(\star)=\star$, $\Theta(s)=s$, $\Theta(b)=b$, $\Theta(b^*)=b^*$, $\Theta(y)=y$ define a graded monoidal superfunctor $\Theta: \GVW \to gr(\sVW)$. 
It is bijective on objects, and full, i.e.\  surjective on morphisms.

\subsection{The Verma module $M(0)$ and the functor $\Psi_n$}

For $n\in \mathbb{N}$, let $\mathfrak{n}_+$ denote the Lie subalgebra of strictly upper triangular matrices, and $\mathfrak{b}$ the Lie subalgebra of lower triangular matrices in $\mathfrak{gl}(n)$. They can be considered as subalgebras of $\mathfrak{gl}(n)= \mathfrak{g}_0\subseteq \fp(n)$ via the inclusion $E_{ij}\mapsto A^-_{ij}$. Consider $\mathbb{C}$ as the trivial representation of $\mathfrak{b}\oplus \mathfrak{g}_{-1}\subseteq \mathfrak{p}(n)$ by letting $A_{ij}^-$ with $i\ge j$ and $C_{ij}^-$ with $i<j$ act on it by $0$. Consider the $\fp(n)$-module
$M(0)=\mathrm{Ind}_{\mathfrak{b}\oplus \mathfrak{g}_{-1}}^{\fp(n)} \mathbb{C}$, the Verma module of highest weight $0$. Using the PBW theorem we can see that, as a vector superspace, this is
$\mathcal{U}(\mathfrak{p}(n))\otimes_{\mathcal{U}(\mathfrak{b}\oplus \mathfrak{g}_{-1})}\mathbb{C}
\cong \Lambda (\mathfrak{g}_1) \otimes \mathrm{S}(\mathfrak{n}_+).$

Consider the filtration on $M(0)$ 
coming from the PBW theorem, i.e.\  given by $\deg(B^+_{ij})=\deg(A^-_{ij})=1$. In particular, $M(0) \otimes V^{\otimes a}$ inherits a filtration (by putting $V$ in degree $0$). In this way, $M(0) \otimes V^{\otimes a}$ can be considered, for any $a \in \NN_0$, as an object in $\FVect$.

\begin{lemma}\label{Psi is filtered}
The superfunctor $\Psi^{M(0)}_n:\sVW \to \pn$ induces (by forgetting the action of $\fp(n)$ on the image of $\Psi^{M(0)}_n$) 
a filtered superfunctor $\Psi_n: \sVW\to \FVect$. 
\end{lemma}
\begin{proof}
The generators $s_i,b_i,b^*_i$ of $\sVW$ have filtered degree $0$, and map under the functor $\Psi_n$ to $\sigma_i, \beta_i, \beta^*_i$ which only act on the $i$-th and $(i+1)$-st tensor factors of $M(0)\otimes V^{\otimes a}$, $1 \leq i \leq a-1$, thus do not change the filtered degree defined on the $0$-th tensor factor $M(0)$. 

The generator $y_k$ has filtered degree $1$ in $\sVW$, and its image under $\Psi_n$ is the operator
$$\Psi_n(y_k)=\sum_{i=0}^{k-1}\Omega_{i k}.$$
For $i =1,\ldots,  k-1$ the operator $\Omega_{i k}$ does not change the filtered degree. 
For $i =0$, the operator $\Omega_{0k}$ acts on $M(0)\otimes V^{\otimes a}$ as
$$\Omega_{0k}=\left( 	\sum_{i,j} A_{ij}^- \otimes A_{ji}^+ 
		- \frac{1}{2} \sum_{i} B_{ii}^+ \otimes C_{ii}^+
		- \sum_{i < j} B_{ij}^+ \otimes C_{ji}^+
		+ \sum_{i < j} C_{ij}^- \otimes B_{ji}^-
 \right)_{0k}.$$
The summands with $C_{ij}^-$, $i<j$ and $A_{ij}^-$, $i\ge j$ in the $0$-th tensor factor preserve the filtered degree. The summands with $B_{ij}^+$, $i\le j$, and $A_{ij}^-$, $i<j$ in the $0$-th tensor factor increase the filtered degree by $1$. Thus, $\Psi_n(y_k)$ acts by increasing the filtered degree by $1$.
\end{proof}

\subsection{The functor $\Phi_n$}
Next, we define a certain graded superfunctor, which will eventually turn out to be $\gr (\Psi_n)$. 

Consider again the vector space $\Lambda (\mathfrak{g}_1) \otimes \mathrm{S}(\mathfrak{n}_+)$, now as a graded superspace with the grading given by $\deg(B_{ij}^+)=\deg(A_{ij}^-)=1$. This gives a grading on $\left( \Lambda (\mathfrak{g}_1) \otimes \mathrm{S}(\mathfrak{n}_+)\right) \otimes V^{\otimes a}$.

Define a functor $\Phi_n: \GVW\to \GVect$ on objects by
$\Phi_n(a)=\left(\Lambda (\mathfrak{g}_1) \otimes \mathrm{S}(\mathfrak{n}_+) \right) \otimes V^{\otimes a}$. In the image, we again label $\Lambda (\mathfrak{g}_1) \otimes \mathrm{S}(\mathfrak{n}_+)$ as the $0$-th tensor factor, and $V\otimes \ldots \otimes V$ as factors $1,2,\ldots, a$. With this convention, set
$\Phi_n(s_i)=\sigma_i$, $\Phi_n(b_i)= \beta_i$, $\Phi_n(b_i^*)= \beta_i^*$, and let
$$\Phi_n(y_k)=\left( 	\sum_{i<j} A_{ij}^- \otimes A_{ji}^+ 
		- \frac{1}{2} \sum_{i} B_{ii}^+ \otimes C_{ii}^+
		- \sum_{i < j} B_{ij}^+ \otimes C_{ji}^+  \right)_{0k},$$
with the action of $A_{ij}^-\in \mathfrak{n}_+$ and of $B_{ij}^+\in \mathfrak{g}_1$ on $\Lambda  (\mathfrak{g}_1) \otimes \mathrm{S}(\mathfrak{n}_+)$ given by multiplication. 

\begin{lemma}\label{Phi is graded}
$\Phi_n:\GVW\to \GVect$ is a well-defined graded superfunctor. 
\end{lemma}
\begin{proof}
This is a direct calculation analogous to Proposition ~\ref{prop:funct} and Lemma ~\ref{Psi is filtered}.
\end{proof}

\begin{lemma}\label{square}
With our fixed $n \in \NN$, the following square strictly commutes:
$$
\xymatrix{
\ar @{} [dr] |{=}
 \GVW \ar[d]_{\Theta} \ar[r]^{\Phi_n}
                 & \GVect        \\
 gr(\sVW) \ar[r]^{gr\Psi_n}   & gr(\FVect)  \ar[u]_G              }$$

That is, $G\circ gr\Psi_n \circ \Theta=\Phi_n$ on all objects and morphisms.
\end{lemma}
\begin{proof}
It clearly strictly commutes on objects, and on the generating morphisms $s_i, b_i,b_i^*$ of degree $0$, so it only remains to check it on $y_k$ of filtered degree $1$. This follows from the proof of Lemma~\ref{Psi is filtered} and from the definition of $\Phi_n$. 
\end{proof}

Define a total ordering $\to $ on the set $[n]\cup [n']$ by saying that $i\to j$ if there is a path (of length at least one) from $i$ to $j$ in the graph
\begin{equation}\label{eqn:graph}
1 \to 2 \to \ldots \to n \to n' \to (n-1)' \to \ldots \to 2' \to 1'. 
\end{equation}
With this we have the following technical tool:
\begin{lemma}\label{graph}
Let $0\not=m\in M(0)$, $i_1,\ldots, i_a \in [n]\cup [n']$, and $1\le k\le a$ be arbitrary.  Then
\begin{align*}
\Phi_n(y_k) (m\otimes v_{i_1}\otimes  v_{i_2}\otimes \ldots \otimes v_{i_a})&=
\sum_{i_k\to j} m_j \otimes v_{i_1}\otimes  \ldots  \otimes v_{i_{k-1}} \otimes v_j \otimes v_{i_{k+1}}  \otimes \ldots   \otimes   v_{i_a}
\end{align*}
for some $m_j\in M(0)$.
Additionally, if $i_k\in [n-1]$, then $m_{i_k+1}=A_{i_k, i_k+1}^- m\ne 0$.
\end{lemma}
\begin{proof} First note that by definition, $\Phi_n(y_k)(m\otimes v_{i_1}\otimes  v_{i_2}\otimes \ldots \otimes v_{i_a})$ equals
\begin{align*}
&\left( 	\sum_{i<j} A_{ij}^- \otimes A_{ji}^+ 
		- \frac{1}{2} \sum_{i} B_{ii}^+ \otimes C_{ii}^+
		- \sum_{i < j} B_{ij}^+ \otimes C_{ji}^+  \right)_{0k} (m\otimes v_{i_1}\otimes  v_{i_2}\otimes \ldots \otimes v_{i_a})=\\
&=\sum_{i<j} A_{ij}^-m \otimes v_{i_1}\otimes  \ldots  \otimes A_{ji}^+  v_{i_k} \otimes \ldots   \otimes   v_{i_a} -\frac{1}{2} \sum_{i} B_{ii}^+ m \otimes v_{i_1}\otimes  \ldots  \otimes C_{ii}^+  v_{i_k} \otimes \ldots   \otimes   v_{i_a} - \\
 &	\quad	- \sum_{i < j} B_{ij}^+m \otimes v_{i_1}\otimes  \ldots  \otimes C_{ji}^+  v_{i_k} \otimes \ldots   \otimes   v_{i_a}. 
\end{align*}
Thus, all summands are of the form $m_j \otimes v_{i_1}\otimes  \ldots  \otimes v_{j} \otimes \ldots   \otimes   v_{i_a}$ for $m_j \in M(0)$. 

To determine the occuring $v_j$, recall that
$A_{ji}^+ = E_{ji} + E_{i' j'}$ and $C_{ji}^+ = E_{j'i} + E_{i' j}$, and therefore we have
\begin{equation}\label{eqn:Aji-action}
A_{ji}^+v_l =\delta_{il}v_j, \quad A_{ji}^+v_{l'} =\delta_{jl}v_{i'} ,  \text{ for } i<j \quad \text{ and } \quad C_{ji}^+v_l =\delta_{il}v_{j'}+\delta_{jl}v_{i'}, \quad C_{ij}^+v_{l'} =0.\end{equation}
In either case, $v_j$ is (possibly a constant multiple of) another standard basis vector, whose index appears strictly to the right of $i_k$ in (\ref{eqn:graph}), thus proving the first claim. 
For the second, it follows from \eqref{eqn:Aji-action} that the only summand transforming $v_{i_k}$ to $v_{i_{k}+1}$ acts by $A_{i_k+1,i_k}^+$ on the $k$-th tensor factor, and thus acts by $A_{i_k,i_k+1}^-$ in the $0$-th tensor factor, replacing $m$ by $A_{i_k,i_k+1}^-m$. 
\end{proof}

\subsection{The key construction}\label{key-construction}
The following construction, associating two vectors $v_d$ and $w_d$ to a diagram $d\in \Hom_{\sVW}(a,b)$, is key to the proof of Theorem~\ref{Thm2} in Section \ref{final-proof-2}. In the special case when $d$ has no dots and has the same number of cups and caps (i.e.\  $d\in \Hom_{\sVW}(a,a)^0$), it specializes to a certain construction from \cite[Section 4]{M}; see Section~\ref{final-proof-1} for details. 

Given a diagram $d\in \Hom_{\sVW}(a,b)^k$ and $n\ge\frac{a+b}{2}+k$, define $v_d\in V^{\otimes a}$ and $w_d\in V^{\otimes b}$ by the following algorithm.

\begin{itemize}
\item[\bf{STEP 0.}] Put an ordering on the strings in $d$ so that caps come first, ordered left to right with respect to their left end; then through strings, ordered left to right with respect to their bottom end; then cups, ordered right to left with respect to their right end. (See for instance (\ref{eqn:Numberstrings}), where the strings are ordered using the set $\{\textcircled{1},\textcircled{2},\textcircled{3},\textcircled{4},\textcircled{5},\textcircled{6},\textcircled{7}\}$ with the usual ordering.)
\item[\bf{STEP 1.}] 
Starting with the smallest cap label, and repeating along the order, label its left end by the minimal $i \in \left[n\right]$ which is bigger than all the labels already assigned. If the cap has $\ell$ dots, label its right end by $i+\ell$.
\item[\bf{STEP 2.}]
Continue with the through strings in the assigned order, and for each, label its bottom end by the minimal $i \in \left[n\right]$ which is bigger than all the labels already assigned. If the through string has $\ell$ dots, label its top end by $i+\ell$.
\item[\bf{STEP 3.}]
For each cup in order, label its right end by the minimal element $i$ of the set $\left[n\right]$ which is bigger than all the labels already assigned. If the cup has $\ell$ dots, label its left end by $i+\ell$.
\item[\bf{STEP 4.}] For each cup and cap, change the right end label from $i$ to $i'$.
\item[\bf{STEP 5.}] Now we have assigned to the bottom of the diagram labels $i_1, i_2,\ldots ,i_a$ and to the top $j_1, j_2,\ldots ,j_b$ for some $i_1, \ldots ,i_a, j_1,\ldots , j_b \in [n] \cup [n'] $. Set
$$v_d =v_{i_1}\otimes v_{i_2}\otimes \ldots \otimes v_{i_a}\in V^{\otimes a}, \quad\text{and}\quad
w_d =v_{j_1}\otimes v_{j_2}\otimes \ldots \otimes v_{j_b}\in V^{\otimes b}.$$
\end{itemize}

\begin{example}\label{big-example-vd-wd}
For instance, for $d=y_1^2 s_2 s_6\beta_3^*\beta_1^* s_3 s_2 \beta_1s_2 y_1^2 y_2 y_4^2 y_6 \in \Hom_{\sVW}(6,8)^8,$

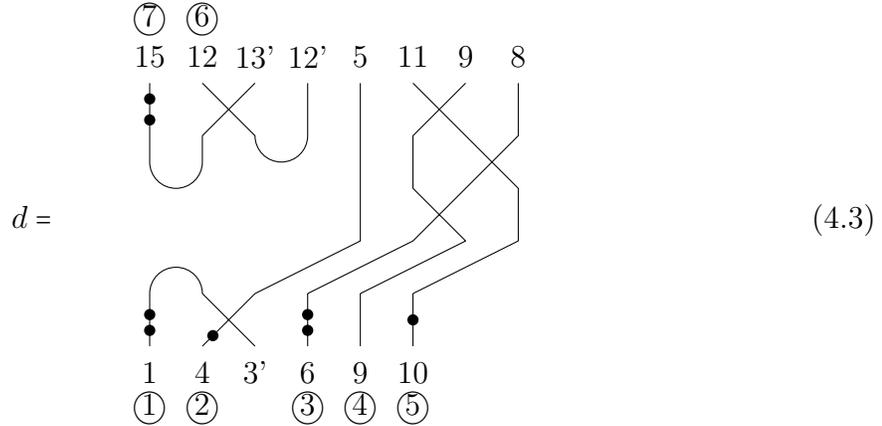
\begin{equation}\label{eqn:Numberstrings}
d=\qquad
\TikZ{[scale=0.7]
\draw 
(0,5) node{} to (0,3.5) node{}
(1,5) node{} to (2,4) node{}
(2,5) node{} to (1,4) node{}
(1,4) node{} to (1,3.5) node{}
(3,5) node{} to (3,4) node{}
(4,5) node{} to (4,4) node{}
(5,5) node{} to (6,4) node{}
(6,5) node{} to (5,4) node{}
(7,5) node{} to (7,4) node{}
(0,4.7) node[fill,circle,inner sep=1.5pt]{}
(0,4.3) node[fill,circle,inner sep=1.5pt]{}
(0,3.5) arc(-180:0:0.5) 
(2,4) arc(-180:0:0.5) 
(4,4) node{} to (4,3) node{}
(5,4) node{} to (5,3) node{}
(6,4) node{} to (7,3) node{}
(7,4) node{} to (6,3) node{}
(4,3) node{} to (4,2) node{}
(5,3) node{} to (6,2) node{}
(6,3) node{} to (5,2) node{}
(7,3) node{} to (7,2) node{}
(0,1) arc(180:0:0.5) 
(4,2) node{} to (2,1) node{}
(5,2) node{} to (3,1) node{}
(6,2) node{} to (4,1) node{}
(7,2) node{} to (5,1) node{}
(0,1) node{} to (0,0) node{}
(1,1) node{} to (2,0) node{}
(2,1) node{} to (1,0) node{}
(3,1) node{} to (3,0) node{}
(4,1) node{} to (4,0) node{}
(5,1) node{} to (5,0) node{}
(0,0.3) node[fill,circle,inner sep=1.5pt]{}
(0,0.6) node[fill,circle,inner sep=1.5pt]{}
(1.2,0.2) node[fill,circle,inner sep=1.5pt]{}
(3,0.3) node[fill,circle,inner sep=1.5pt]{}
(3,0.6) node[fill,circle,inner sep=1.5pt]{}
(5,0.5) node[fill,circle,inner sep=1.5pt]{}
(0,-0.5) node{1}
(0,-1.2) node{\textcircled{1}}
(2,-0.5) node{3'}
(1,-0.5) node{4}
(1,-1.2) node{\textcircled{2}}
(4,5.5) node{5}
(3,-0.5) node{6}
(3,-1.2) node{\textcircled{3}}
(7,5.5) node{8}
(4,-0.5) node{9}
(4,-1.2) node{\textcircled{4}}
(6,5.5) node{9}
(5,-0.5) node{10}
(5,-1.2) node{\textcircled{5}}
(5,5.5) node{11}
(3,5.5) node{12'}
(1,5.5) node{12}
(1,6.2) node{\textcircled{6}}
(2,5.5) node{13'}
(0,5.5) node{15}
(0,6.2) node{\textcircled{7}}
;}\end{equation}
we get $v=v_1\otimes v_4\otimes v_{3'}\otimes v_6\otimes v_9\otimes v_{10}\in V^{\otimes 6}$, and $w_d=v_{15}\otimes v_{12}\otimes v_{13'}\otimes v_{12'}\otimes v_{5}\otimes v_{11}\otimes v_{9}\otimes v_8 \in V^{\otimes 8}$.
\end{example}

\begin{remark} \label{remark:theKeyconstruction}
The largest label is always $\frac{a+b}{2}+k$. We require $n\ge \frac{a+b}{2}+k$ in order to be able to realize $v_{\frac{a+b}{2}+k}\in V=\mathbb{C}^{n|n}$ in STEP 5.

The ordering in STEP 0 could be changed, as long as all caps come first, then all through strings, then all cups. This changes the vectors $v_d$ and $w_d$, but preserves the important features of the construction. 

Observe also that if $i',j'\in [n']$ are labels with $i'$ at the bottom, $j'$ at the top, then $i<j$. 
\end{remark}

\subsection{The key lemma}

The proof of linear independence relies on the observation that the vectors $v_d$ and $w_d$ can be used to distinguish diagrams  in $S^k_{a,b}$.  

Namely, the standard basis $v_1,\ldots , v_n,v_{1'}, \ldots, v_{n'}$ of $V$ induces a standard basis $B_b$ of $V^{\otimes b}$. For any vector $z\in M(0)\otimes V^{\otimes b}$ and any standard basis vector $w\in B_b$ we denote by $\left< w \mid z\right> \in M(0)$ the coefficient of $z$ in this standard basis. In other words, 
$$z=\sum_{w\in B_b} \left< w \mid z\right>\otimes w.$$

\begin{lemma}\label{Key lemma}
Let $a, b, k \in \NN_0$. For any $d, d' \in S_{a,b}^k$, we have $\left< w_d \mid \Phi_n(d') v_d\right> \ne 0$ iff  $d=d'$.
\end{lemma}

\begin{proof}

\fbox{$\Leftarrow$}
We repeatedly use the second part of  Lemma~\ref{graph}. 

Consider a cap with $\ell$ dots on it, and the edges labelled $i$ and $(i+\ell)'$. By Lemma ~\ref{graph}, applying the $\ell$ dots replaces $v_d$ by a linear combination of vectors which have the tensor factor $v_i$ of $v_d$ replaced by some $v_{j}$'s with $i\to j$, such that the path in \eqref{eqn:graph} from $i$ to $j$ has length at least $\ell$.
Exactly one such summand will give a non-zero contribution when such a $v_j$ is paired with $v_{(i+\ell)'}$ via $\beta$; namely, the one with $j= i+\ell$. Applying this dotted cap transforms the $0$-th tensor factor, say $m$, into the factor $A_{i,i+1}^-A_{i+1,i+2}^-\ldots A_{i+\ell -1,i+\ell }^- m$. 

Next, consider a through string with $\ell $ dots and labels $i$ and $i+\ell$. It prescribes the order of some superswaps of tensor factors of $v_d$. After applying the $\ell$ dots,  $v_d$ is replaced by a linear combination of vectors which have the tensor factor $v_i$ of $v_d$ replaced by some $v_{j}$ with $i\to j$, for which the path in \eqref{eqn:graph} from $i$ to $j$ has length at least $\ell$.
Reading off the coefficient of $w_d$ manifests itself in the tensor factor corresponding to this string to reading off the coefficient of $v_{i+\ell}$. The only summand with a non-zero contribution is the one with $j={i+\ell}$; in effect the $0$-th tensor factor got acted on by $A_{i,i+1}^-A_{i+1,i+2}^-\ldots A_{i+\ell -1,i+\ell}^-$. 

Finally, consider a cup with $\ell$ dots and labels $i+\ell$ and $i'$. The $\beta^*$ corresponding to this cup produced $\sum_{j}(v_j\otimes v_{j'}-v_{j'} \otimes v_{j})$; applying the $\ell$ dots on the left end of it and reading off the coefficient of $v_{i+\ell}\otimes v_{i'}$ (as prescribed by $\left< w_d \mid\cdot \;\right>$) gives exactly one summand with a non-zero contribution. The effect on the $0$-th tensor factor is action by $A_{i,i+1}^-A_{i+1,i+2}^-\ldots A_{i+\ell-1,i+\ell}^-$. Thus,  $\left< w_d \mid \Phi_n(d')v_d \right>$ is, up to a possible sign, equal to
\begin{align*}
\prod_{\TikZ{[scale=0.4]
\draw 
(-0.1,0.5) arc(180:0:0.5)
(-0.1,0.5) node[fill,circle,inner sep=1.5pt]{}
(-0.4,0) node{\tiny{i}}
(1.1,-0) node{\tiny{(i+$\ell$)'}}
;}} A_{i,i+1}^-\ldots A_{i+\ell-1,i+\ell}^- \cdot 
\prod_{\TikZ{[scale=0.4] 
\draw 
(0,0) node{} to (0,1) node{}
(0,0) node[fill,circle,inner sep=1.5pt]{}
(-0.4,0) node{\tiny{i}}
(-0.7,1) node{\tiny{i+$\ell$}}
;}} A_{i,i+1}^-\ldots A_{i+\ell-1,i+\ell}^- \cdot 
 \prod_{\TikZ{[scale=0.4]
\draw 
(0,0) arc(-180:0:0.5)
(0,0) node[fill,circle,inner sep=1.5pt]{}
(-0.4,0.5) node{\tiny{i+$\ell$}}
(1.2,0.5) node{\tiny{i'}}
;}} A_{i,i+1}^-\ldots A_{i+\ell-1,i+\ell}^- \ne 0,
\end{align*}
where the factors are given by the shape and the assigned labels of $d$.

\fbox{$\Rightarrow$}
Let $d' \in S_{a,b}^k$ be any diagram for which
$\left< w_d \mid  \Phi_n(d')v_d \right> \ne 0.$ We first recover the underlying connector $P(d')$ from the labelling of $d$. 

Consider any cap in $d'$. By Lemma~\ref{graph} and the ordering $\to$, the dots increase indices $i\in [n]$ or replace them by $j'\in [n']$, and they decrease $j'\in [n']$. From that, and the facts $\left< w_d \mid \Phi_n(d')v_d\right> \ne 0$ and $\beta(v_i,v_j)=\beta(v_{i'},v_{j'})=0$, $\beta(v_i,v_{j'})=\delta_{ij}$, it follows that a cap in $d'$ can connect two points which are labelled in $d$ by an (unordered) pair of the form $\{i,j\}$ or $\{i,j'\}$ with $i\le j$. 

Next, consider any cup in $d'$. Note that $\beta^*(1)=\sum_{i}(v_i\otimes v_{i'}-v_{i'}\otimes v_i)$, and that subsequent application of dots increases $i\in [n]$ or replaces it by $j'\in [n']$, and decreases $j'\in [n']$. Hence $\left< w_d \mid \Phi_n(d')v_d \right> \ne 0$ implies that a cup in $d'$ can only connect those pairs of points in $d$ labelled by $\{i',j'\}$, or by $\{i,j'\}$ with $i\ge j$. 

Finally, consider any through string in $d'$. The possibilities for its labels (bottom and top) are then, by Lemma~ \ref{graph} and $\left< w_d \mid \Phi_n(d')v_d \right> \ne 0$, given by ordered pairs of the form $(i,j')$, or of the form $(i,j)$ with $i\le j$, or of the form $(i',j')$ with $i\ge j$. However, the last of these is not possible by Remark \ref{remark:theKeyconstruction}, so the remaining possibilities for the bottom and top labels of a through string are $(i,j')$ and $(i,j)$ with $i\le j$.

For any diagram $d''$, let $\cap(d'')$ denote the number of caps of $d''$; $\cup(d'')$ the number of cups, and $t(d'')$ the number of through strings. By the above analysis, 
all labels $i'\in [n']$ on the bottom are on caps in $d'$, so
\begin{align}\label{eqn:ineq1}
\cap(d')\ge \# \textrm{ labels } j' \in [n'] \textrm{ at the bottom}= \cap(d).
\end{align}
As every cup in $d'$ has at least one label of type $ j' \in [n'] $, we also see that 
\begin{align}\label{eqn:ineq2}
\cup(d')\le \# \textrm{ labels } j' \in [n'] \textrm{ at the top}= \cup(d).
\end{align}
We get a sequence of inequalities 
\begin{align*}
t(d')= a-2\cap(d') \stackrel{(\ref{eqn:ineq1})}{\le} a-2\cap(d) = t(d)=
 b-2\cup(d) \stackrel{(\ref{eqn:ineq2})}{\le} b-2\cup(d') = t(d').
\end{align*}
This implies that \eqref{eqn:ineq1} and \eqref{eqn:ineq2} are equalities, and moreover
\begin{align}\label{eqn:ineq3}
\cap(d')= \cap(d), \quad \cup(d')= \cup(d), \quad t(d')=t(d). 
\end{align}
So, $d$ and $d'$ have the same number of cups, of caps and of through strings. 

Next, we reconstruct the caps of $d'$. We saw in \eqref{eqn:ineq1}, \eqref{eqn:ineq3} that any label $j'\in [n']$ on the bottom of the diagram $d'$ needs to be on a cap, and all caps have exactly one label of type $j'\in [n']$. The other end of that cap is labelled by some $i\in [n]$ with $i\le j$. 
Starting from the smallest bottom label of type $j'\in [n']$, there is exactly one label at the bottom of type $i\in [n]$ with $i\le j$, so these two labels must be joined by a cap in $d'$. To get the non-vanishing of the action of the dots composed with $\beta$ prescribed by this cap, this cap needs by Lemma~\ref{graph}  to have at most $j-i$ dots in $d'$. (It has exactly $j-i$ dots in $d$). Proceed with the next smallest label of type $j'\in [n']$, noticing that there is exactly one unpaired label $i$ with $i\le j$, and pair them. After doing this for all $j' \in \left[n'\right]$ on the bottom, we see that the connectors $P(d')$ and $P(d)$ have the same pairing of the points given by caps, and every cap in $d'$ has at most as many dots as the corresponding cap in $d$.

Next, we recover the cups. By (\ref{eqn:ineq2}) and (\ref{eqn:ineq3}), every label of type $ j' \in [n']$ needs to be on an end of a cup, whereas the other end is labelled by some $i\in [n]$ with $j\le i$, and which has at most $i-j$ dots. By STEP 0 the cups come last, so there is exactly one such pairing of points on the top. So, $P(d')$ and $P(d)$ also have the same pairing of the points given by cups, and every cup in $d'$ has at most as many dots as the corresponding cup in $d$. Finally, all remaining unassigned labels are of type $i\in [n]$, and there is exactly one pairing such that the bottom label is smaller than the top label. So, the connectors $P(d')$ and $P(d)$ have the same pairing of the points given by through strings, and every through string in $d'$ has at most as many dots as the corresponding through string in $d$.

Therefore, $P(d)=P(d')$. As the underlying undotted diagrams of $d$ and $d'$ are both in $S_{a,b}$, they are the same. Finally, as $d'$ has at most as many dots as $d$ on every string, and they have the same total number of dots, we conclude that $d'=d$. 
\end{proof}

\begin{example}
For the diagram $d$ from Example \ref{big-example-vd-wd}, 
$$\left< w_d \mid \Phi_n(d)v_d \right> =A_{12}^-A_{23}^-A_{45}^-A_{67}^-A_{78}^-A_{10,11}^-A_{13,14}^-A_{14,15}^- \in \mathcal{U}(\mathfrak{n}_-)=M(0).$$

\end{example}

\subsection{Proof of Theorem~\ref{Thm2}}\label{final-proof-2}

In this section we will finally prove the linearly independence of $S_{a,b}^\bullet$, thus proving Theorem~\ref{Thm2}. We start by proving it in the graded setting. 

\begin{lemma}\label{Phi_n injecitve}
Given $a, b, k\in \mathbb{N}_0$, and $n\ge \frac{a+b}{2}+k$, the map 
$$\Phi_n:\Hom_{\GVW}(a,b)^k\longrightarrow \Hom_{\GVect}(M(0)\otimes V^{\otimes a}, M(0)\otimes V^{\otimes b})^k$$ maps the set $S_{a,b}^k$ to a linearly independent set. Thus, $S_{a,b}^k$ is linearly independent in $\Hom_{\GVW}(a,b)^k$, and $\Phi_n$ is injective on $\Hom_{\GVW}(a,b)^k$.  
\end{lemma}
\begin{proof}
Assume there are some $\alpha_{d'} \in \mathbb{C}$ such that
$\sum_{d' \in S_{a,b}^k} \alpha_{d'} \Phi_n(d')=0$.
For any $d \in S_{a,b}^k$, applying both sides of the above equation to the vector $v_d$, reading off the coefficient of $w_d$, and applying Lemma~\ref{Key lemma}, we get
$\alpha_d=0.$

So, the set $\{ \Phi_n(d) \mid d\in S_{a,b}^k\}$ is linearly independent. From that it follows that $S_{a,b}^k$ is linearly independent in $\Hom_{\GVW}(a,b)^k$. It is also a spanning set for $\Hom_{\GVW}(a,b)^k$ by Lemma~\ref{Phi is graded}, so $\Phi_n$ is injective on $\Hom_{\GVW}(a,b)^k$.
\end{proof}

\begin{corollary}\label{Basis-gsVW}
For all $a,b\in \mathbb{N}_0$, the set $S_{a,b}^\bullet$ is a basis of $\Hom_{\GVW}(a,b)$.
\end{corollary}

\begin{lemma}
\label{lemmathm2}
For all $a,b\in \mathbb{N}_0$, the set $S_{a,b}^\bullet$ is linearly independent in $\Hom_{\sVW}(a,b).$
\end{lemma}
\begin{proof}
Assume there is a nontrivial relation among elements of $S_{a,b}^\bullet$ in $\Hom_{\sVW}(a,b).$ As this is a filtered category, the highest order terms (of degree $k$) in this relation give a nontrivial relation among the elements of $S_{a,b}^k$ in $\Hom_{gr(\sVW)}(a,b)$. Thus, it is enough to prove that the set $S_{a,b}^k$ is linearly independent in $\Hom_{gr(\sVW)}(a,b)$ for each $k$. 

Set $n=\frac{a+b}{2}+k$ and consider the square
$$
\xymatrix{
 \Hom_{\GVW}(a,b)^k \ar@{->>}[d]_{\Theta} \ar@{^{(}->}[r]^<<<<<<{\Phi_n}
                 & \Hom_{\GVect}(M(0)\otimes V^{\otimes a}, M(0)\otimes V^{\otimes b})^k        \\
 \Hom_{gr(\sVW)}(a,b)^k \ar[r]^<<<<<{gr\Psi_n}   & \Hom_{gr(\FVect)}(M(0)\otimes V^{\otimes a}, M(0)\otimes V^{\otimes b})^k  \ar[u]_G     }$$

The map $\Phi_n$ is injective by Lemma~\ref{Phi_n injecitve}, and the diagram strictly commutes by Lemma~\ref{square}. Thus, $\Theta$ is injective. It is surjective by Section \ref{Theta}, so it is an isomorphism of superspaces. 

In particular, $\Theta$ maps the basis $S_{a,b}^k$ of $\Hom_{\GVW}(a,b)^k$ to a basis in $\Hom_{gr(\sVW)}(a,b)^k$ which by construction is $S_{a,b}^k$.
\end{proof}
\begin{corollary}$\Theta:\GVW\to gr(\sVW)$ is a graded isomorphism. \label{theta-iso}
\end{corollary}
\begin{corollary}
For any $a, b, k$, and $n\ge  \frac{a+b}{2}+k$, the map $\Psi_n$ is injective on $\Hom_{\sVW}(a,b)^{\le k }$.
\end{corollary}
Theorem~\ref{Thm2} now follows directly from Proposition~\ref{oVW-span} and Lemma~\ref{lemmathm2}.

\subsection{A basis theorem for $\sBr$ as a special case} \label{final-proof-1}

Theorem~\ref{Thm1} now follows immediately by realizing the supercategory $\sBr$ as the $0$-th filtration piece of the supercategory $\sVW$. 

\begin{proof}[Proof of Theorem 1]
Consider the functor $I: \sBr\to \sVW$ which is the identity on objects and interprets undotted diagrams as dotted diagrams with zero dots. For every $a$ and $b$, $I:\Hom_{\sBr}(a,b)\to \Hom_{\sVW}(a,b)$ maps the spanning set $S_{a,b}$ to the set $S_{a,b}^0$, which by Theorem~\ref{Thm2} is a basis of $\Hom_{\sVW}(a,b)^0$. Thus, the set $S_{a,b}$ is a basis of $\Hom_{\sBr}(a,b)$.
\end{proof}

\begin{remark}
The functor $\Psi_n\circ I:\sBr\to \FVect$ can be decomposed as
$\Psi_n\circ I=J_n\circ \Psi_n^{\mathbb{C}}$
where $\Psi_n^{\mathbb{C}}:\sBr\to \mathcal{V}ect$ is given on objects by $\Psi_n^{\mathbb{C}}(a)=V^{\otimes a}$ and the expected map on morphisms, 
and $J_n:\mathcal{V}ect \to \FVect$, is given by $J_n(W)=M(0)\otimes W$. The functor $\Psi_n^{\mathbb{C}}$ appears in \cite{M}. It is shown there that when $n\ge a$, $\Psi_n^{\mathbb{C}}:\Hom_{\sBr}(a,a)\to \Hom_{\fp(n)}(V^{\otimes a}, V^{\otimes a})$ maps $S_{a,a}$ to a linearly independent set, thus proving that $S_{a,a}$ is a basis, and that $\Psi_n^{\mathbb{C}}$ is injective on $\Hom_{\sBr}(a,a)$. It is also proved that $\Psi_n^{\mathbb{C}}$ is surjective, so $\mathrm{End}_{\sBr}(a) \cong \mathrm{End}_{\pn}(V^{\otimes a})$ for $a \leq n$ (see \cite[Theorem 4.5]{M}). 
\end{remark}

\begin{remark}
Clearly $\Psi^{\C}_n$ is not injective if $n<a$ since it is not injective when restricted to the symmetric group $S_a$. The question of surjectivity of the functors $\Psi_n^M$ for different modules $M$ is interesting and so far not understood. One would need to better understand the combinatorics of decomposition numbers in $\pn$ or category $\mathcal{O}(\fp(n))$. To our knowledge, only the decomposition numbers of the finite dimensional (thick and thin) Kac modules are known, see \cite{BDEHHILNSS}. Even in these cases, a precise surjectivity statement is so far not available. Based on explicitly calculated examples, we expect a more involved behaviour than in the $\fgl(n|n)$ case, see \cite{BS4}.
\end{remark}

\section{The affine VW superalgebra $\sVW_a$ and its centre} 
\label{centre}

We fix $a\ge 2 \in \mathbb{N}$ for the whole section, and study the affine VW superalgebra $\sVW_a=\End_{\sVW}(a)$. The results from the previous section show that the algebra $\sVW_a$ is a PBW deformation of the algebra $\GVW_a$, in the sense that $\sVW_a$ is a filtered algebra, and $\gr(\sVW_a)=\GVW_a$. For $\hbar$ a parameter, the Rees construction gives the algebra $A_{\hbar}$ over $\mathbb{C}[\hbar]$, such that its specializations at $\hbar=0$ and $\hbar=1$ are precisely $A_1= \sVW_a$ and $A_0=\GVW_a$. We then use Theorem \ref{Thm2} to  describe the center of the $\mathbb{C}[\hbar]$-algebra $A_{\hbar}$, and all its specializations $A_t$ for any $t\in \mathbb{C}$; in particular we find the centre of $\sVW_a$ and $\GVW_a$. 
We refer e.g.  to \cite{BG}, \cite{HSS}, \cite{Schedler}, \cite{SW} for the general theory.

\subsection{The algebras $A_{\hbar}$}
\label{secAbar}

We first define a $\mathbb{C}[\hbar]$-algebra $A_{\hbar}$ and its specializations $A_t$ at $t\in \mathbb{C}$ directly using generators and relations. 

\begin{definition}\label{def:Ahbar}
Let $A_{\hbar}$ be the superalgebra over $\mathbb{C}[\hbar]$
with generators
	$$s_i,\; e_i,\; y_j\quad 1 \leq i \leq a-1, \; 1 \leq j\leq a$$
	where $\overline{s_i}=\overline{e_i}=\overline{y_j}=0$, subject to the relations:
	
	\begin{multicols}{2}
		\begin{enumerate}[(VW1)]
			
			\item Involutions: $s_i^2=1 \text{ for } 1\leq i <a.$
			
			\item Commutation relations:\label{comm}
			\begin{enumerate}[label=(\roman*)]
				\item $s_ie_j=e_js_i$ if $|i-j|>1$,
				\item $e_ie_j=e_je_i$ if $|i-j|>1$,
				\item $e_iy_j=y_je_i$ if $j\neq i, i+1$,
				\item $y_iy_j=y_jy_i$ for $1\leq i,j\leq a$.
			\end{enumerate}
			
			\item Affine braid relations:\label{affbraid}
			\begin{enumerate}[label=(\roman*)]
				\item $s_is_j=s_js_i$ if $|i-j|>1$,
				\item $s_is_{i+1}s_i=s_{i+1}s_is_{i+1}$ 
				
				for $1\leq i \leq a-1$,
				\item $s_iy_j=y_js_i$ if $j\neq i, i+1$.
			\end{enumerate}
			
			\item Snake relations:
			\begin{enumerate}[label=(\roman*)]
				\item $e_{i+1}e_ie_{i+1}=- e_{i+1}$,
				\item $e_ie_{i+1}e_i=- e_i$
				
				for  $1\leq i\leq a-2$.
			\end{enumerate}

			\item Tangle and untwisting relations:\label{tang}
			\begin{enumerate}[label=(\roman*)]
				\item \label{twist}$e_is_i=e_i$ and $s_ie_i=-e_i$
				
				for $1\leq i\leq a-1$,
				
				\item $s_ie_{i+1}e_i= s_{i+1}e_i$,
				\item $s_{i+1}e_ie_{i+1}=-s_ie_{i+1}$,
				\item $e_{i+1}e_is_{i+1}= e_{i+1}s_i$,
				\item $e_ie_{i+1}s_i=-e_is_{i+1}$
				
				for $1\leq i \leq a-2$.
				
			\end{enumerate}
			
			\item \label{idemp}Idempotent relations:
			
			$e_i^2=0$ for $1\leq i \leq a-1$.
			
			\item Skein relations:\label{skein}
			\begin{enumerate}[label=(\roman*)]
				\item $s_iy_i-y_{i+1}s_i=-\hbar e_i-\hbar$,
				\item $y_is_i-s_iy_{i+1}= \hbar e_i-\hbar$
				
				for $1\leq i\leq a-1.$
			\end{enumerate}
			
			\item \label{unwrap} Unwrapping relations: 
			
			$e_1y_1^ke_1=0 \text{ for }k \in \NN$.
			
			\item\label{sym} (Anti)-Symmetry relations:
			\begin{enumerate}[label=(\roman*)]
				\item $e_i(y_{i+1}-y_{i})=\hbar e_i$,
				\item $(y_{i+1}- y_{i})e_i= -\hbar e_i$
				
				for $1\leq i\leq a-1$.	
			\end{enumerate}			
			
		\end{enumerate}
	\end{multicols}
	
For any $t\in \mathbb{C}$, let $A_{t}$ be the quotient of $A_{\hbar}$ by the ideal generated by $\hbar-t$. 
\end{definition}

\begin{remark}
	The above set of relations is not minimal. For instance, relations \ref{idemp} and \ref{unwrap} can be deduced from \ref{tang}\ref{twist} and \ref{sym}.
\end{remark} 
	
As a $\mathbb{C}[\hbar]$-algebra, $A_{\hbar}$ is filtered by $\deg(y_i)=1$, $\deg(s_i)=\deg(e_i)=0.$
Considered as a $\mathbb{C}$-algebra, $A_{\hbar}$ can be given a grading by setting $\deg(y_i)=\deg(\hbar)=1$, $\deg(s_i)=\deg(e_i)=0.$ Interpreting  $s_i$ , $e_i=b_i*b_i$, $y_i$ as diagrams as in Section~\ref{sectionone}, the elements of  $A_{\hbar}$ and $A_t$ can be written as linear combinations of dotted diagrams with $a$ bottom points and $a$ top points.

\begin{lemma}
The set $S_{a,a}^{\bullet}$ is a spanning set for $A_{\hbar}$ and $A_t$ for any $t$. 
\end{lemma}
\begin{proof}
Using the braid, snake and untwisting relations (analogous to  (R1)-(R4)) in $A_{\hbar}$ or $A_t$ we see that every element of $S_{a,a}^{\bullet}$ gives rise to a well-defined element of $A_{\hbar}$, respectively $A_t$. Then we can repeat the proof that $S_{a,a}^{\bullet}$ spans $\sVW_a$ for these algebras. 
\end{proof}

 \begin{proposition}
 \begin{enumerate}[(a)]
 \item The assignments  $\varphi_1(y_i)=y_i$, $\varphi_1(s_i)=s_i$ and $\varphi_1(e_i)=b_i^*b_i$  define an isomorphism of algebras $\varphi_1:A_1 \to \sVW_a$. 
  \item  The assignments  $\varphi_0(y_i)=y_i$, $\varphi_0(s_i)=s_i$ and $\varphi_0(e_i)=b_i^*b_i$  define an isomorphism of algebras $\varphi_0:A_0 \to \GVW_a$. 
  \item For any $t\ne 0$, the assignments  $\psi_t(y_i)=ty_i$, $\psi_t(s_i)=s_i$ and $\psi_t(e_i)=e_i$  define an isomorphism of algebras $\psi_t:A_t \to A_1$.
   \item The set $S^\bullet_{a,a}$ is a $\mathbb{C}$-basis of $A_t$ for any $t$, and a $\mathbb{C}[\hbar]$-basis of $A_{\hbar}$.  
 \end{enumerate}
 \end{proposition}
  \begin{proof}
 \begin{enumerate}[(a)]
\item One checks directly that $\varphi_1$ can be extended to an algebra homomorphism by checking that all relations from Definition \ref{def:Ahbar} hold in $\sVW_a$. To see surjectivity, consider an arbitrary element $b$ of $\sVW_a$, and let us construct its preimage. Assume without loss of generality that $b=p(y_1,\ldots, y_a) \, d \, q(y_1,\ldots, y_a)$ for some monomials $p,q$, and some undotted diagram $d$. If $d$ has $c$ cups, then it also has $c$ caps, and can be written in the form $d=\sigma_1 (b_1^*b_1)(b_2^*b_2)\ldots (b_c^*b_c)\sigma_2$ for some permutations $\sigma_1, \sigma_2$. Thus, $b=p\sigma_1 (b_1^*b_1)(b_2^*b_2)\ldots (b_c^*b_c)\sigma_2  q=\varphi_1(p \sigma_1 e_1e_2\ldots e_c\sigma_2  q)$. So, $\varphi_1$ is a surjective homomorphism mapping a spanning set to a basis, so it is an isomorphism. 

\item Analogous to (a). 
\item A direct check of the relations shows that this assignment extends to an algebra homomorphism for any $t\in \mathbb{C}$. For $t\ne 0$, the inverse is given by $\psi_t^{-1}(y_i)=\frac{1}{t}y_i$, $\psi_t^{-1}(s_i)=s_i$ and $\psi_t^{-1}(e_i)=e_i$.  
\item 
For any $t\ne 0$, $S_{a,a}^\bullet$ is a basis of $\sVW_a$ by Theorem~\ref{Thm2}, so by (a) and (c) above it is also a basis of $A_t\cong A_1\cong \sVW_a$. For $t=0$, $S_{a,a}^\bullet$ is a basis of $\GVW_a\cong A_0$ by Corollary~\ref{Basis-gsVW}. Assume there is a relation among the elements of $S_{a,a}^{\bullet}$ in  $A_{\hbar}$, with coefficients in $\mathbb{C}[\hbar]$. Evaluating at some $t\in \mathbb{C}$ for which not all coefficients vanish, we get a relation in $A_t$, which is impossible. So, $S_{a,a}^{\bullet}$ is also a basis of $A_{\hbar}$. \qedhere
 \end{enumerate} 
\end{proof}

 \subsection{The Rees construction}
 Let $B=\bigcup_{k\ge 0}B^{\le k}$ be a filtered  $\mathbb{C}$-algebra. The \emph{Rees algebra} of $B$ is the $ \mathbb{C}[\hbar]$-algebra $\op{Rees}(B)$, given as a $\mathbb{C}$-vector space by $\op{Rees}(B)=\bigoplus_{k\geq 0}B^{\le k}\hbar^k$, with multiplication and the $\hbar$-action both given by $(a\hbar^i)(b\hbar^j)=(ab)\hbar^{i+j}$ for $a\in B^{\le i}$, $b\in B^{\le j}$, and $ab\in B^{\le i+j}$ the product in $B$.  It is graded as a $\mathbb{C}$-algebra by the powers of $\hbar$.

\begin{lemma}\label{lem:Rees-basis}
Let $\bigcup_{i\ge 0} S_i$ be a basis of $B$ compatible with the filtration, in the sense that the $S_i$'s are pairwise disjoint, and $\bigcup_{i= 0}^k S_i$ is a basis of $B^{\le k}$. Then $\bigcup_{i\ge 0} S_i\hbar^i$ is a $\mathbb{C}[\hbar]$-basis of $\op{Rees}(B)$. 
\end{lemma}
\begin{proof}
The set $\bigcup_{i= 0}^k S_i$ is a $\mathbb{C}$-basis of $B^{\le k}$, so  
$\bigcup_{i= 0}^k S_i\hbar^k$ is a $\mathbb{C}$-basis of $B^{\le k}\hbar^k$, and then $\bigcup_{k\ge 0} \bigcup_{i= 0}^k S_i\hbar^k$ is a $\mathbb{C}$-basis of $\op{Rees}(B)$. On the other hand,  $\bigcup_{k\ge 0} \bigcup_{i= 0}^k S_i\hbar^k=\bigcup_{i\ge 0} \bigcup_{k\ge i} S_i\hbar^k=\bigcup_{i\ge 0} \bigcup_{j\ge 0} S_i\hbar^{i+j}=\bigcup_{j\ge 0} \hbar^j (\bigcup_{i\ge 0} S_i\hbar^{i})$. Thus, the set $\bigcup_{i\ge 0}S_i\hbar^i$ is a $\mathbb{C}[\hbar]$-basis of $\op{Rees}(B)$. 
 \end{proof}

For any algebra $B$, let $Z(B)$ denote the centre of $B$. 

\begin{lemma}
\label{ZRees}
$Z(\op{Rees}(B))=\op{Rees}(Z(B))$.
\end{lemma}
\begin{proof}
The centre of $B$ inherits the filtration of $B$, and $\op{Rees}(Z(B))$ embeds naturally into $\op{Rees}(B)$, with the image contained in $Z(\op{Rees}(B))$. To see the other inclusion, assume $c$ is central in $\op{Rees}(B)$.  Without loss of generality $c$ is of homogeneous graded degree $i$, so $c=b \hbar^i$ for some $b \in B^{\le i}$. This shows that $b$ is a central in $B$, proving the claim. 
\end{proof}

\begin{lemma} \label{algiso}
There is an isomorphism of $\mathbb{C}[{\hbar}]$-algebras
$\op{Rees}(A_1)\cong A_{\hbar}$.
\end{lemma}
\begin{proof}
The map $A_{\hbar}\to \op{Rees}(A_1)$ is given on generators by $y_i\mapsto \hbar y_i$, $s_i \mapsto s_i$, $e_i \mapsto e_i$. It is verified to be a morphism of algebras by directly comparing relations, and it is an isomorphism as it maps the basis $S_{a,a}^\bullet$ to the basis $S_{a,a}^\bullet$. 
\end{proof}

\subsection{The centre is a subalgebra of the symmetric polynomials}

We now start computing the centre of $A_{\hbar}$, and show that $Z(A_{\hbar})\subseteq \mathbb{C}[\hbar][y_1,\ldots , y_a]^{S_a}$.

\begin{lemma}\label{centre-poly}
For $f \in A_{\hbar}$, the following are equivalent: 
\begin{enumerate}[(a)] 
\item $fy_i=y_if$ for all $i\in [a]$;
\item $f\in \mathbb{C}[\hbar][y_1,\ldots , y_a]$.
\end{enumerate}
\end{lemma}
\begin{proof}
Because of relation \ref{comm} (iv), only $(a)\Rightarrow (b)$ requires proof.

Assume that $f\notin\mathbb{C}[\hbar][y_1,\ldots , y_a]$. That means that the expansion of $f$ in the basis $S^\bullet_{a,a}$ contains at least one dotted diagram whose underlying undotted diagram is not the identity $1_a$.

Assume that this expansion of $f$ in the basis $S^\bullet_{a,a}$ contains at least one dotted diagram with a cup. Label the top and bottom endpoints of strings $1,\ldots, a$ from left to right. Among all diagrams with a cup, choose $d$ with a maximal number of dots on a cup; say that this cup is connecting $i$ and $j$, and has $k$ dots on it. Then $y_if$, written in the basis $S^\bullet_{a,a}$, contains at least one diagram with a cup and $k+1$ dots on it (namely, $y_id$). On the other hand, $fy_i$ contains no diagrams with $k+1$ dots on a cup, so $y_if\ne fy_i$. 

Now assume that the expansion of $f$ in the basis $S^\bullet_{a,a}$ contains no diagrams with cups, and consequently no diagrams with caps. Then it contains at least one dotted diagram with a through strand connecting differently labelled points at the top and the bottom. Among all such diagrams and all such strings, choose $d$ with a maximal number of dots on such a string; say the string is connecting $i$ at the top of the diagram and $j$ at the bottom, $i\ne j$, and it has $k$ dots on it. Then $y_if$, written in the basis $S^\bullet_{a,a}$, contains at least one diagram with a string connecting $i$ and $j$ and with $k+1$ dots on it, while $fy_i$ contains no such diagrams as $i\ne j$. So, $y_i f\ne fy_i$. 
\end{proof}

In particular, $Z(A_{\hbar})\subseteq \mathbb{C}[\hbar][y_1,\ldots , y_a]$. The following lemma shows that $Z(A_{\hbar})$ is in fact a subalgebra of the symmetric polynomials $\mathbb{C}[\hbar][y_1,\ldots , y_a]^{S_a}.$
\begin{lemma}\label{centre-sym}
Let $f \in \mathbb{C}[\hbar][y_1,\ldots , y_a]\subseteq A_{\hbar}$ and $1\le i\le a-1$.
\begin{enumerate}[(a)]
\item
If $fs_i=s_if$, then $f(y_1,\ldots, y_i,y_{i+1},\ldots, y_a)=f(y_1,\ldots, y_{i+1},y_i,\ldots, y_a)$.
\item For the special value $\hbar=0$, the converse also holds: if $f(y_1,\ldots, y_i,y_{i+1},\ldots, y_a)=f(y_1,\ldots, y_{i+1},y_i,\ldots, y_a)$, then $fs_i=s_if$ in $A_0$.
\end{enumerate}
\end{lemma}
\begin{proof} It is enough to prove this for $a=2$.
\begin{enumerate}[(a)]
\item By Lemma~\ref{AMR2.3}, the expansion of $fs_1$ in the basis $S^\bullet_{a,a}$ is
\begin{equation}\label{eqn:fsi}f(y_1,y_2)s_1=s_1f(y_2,y_1)+\hbar \sum_{i,j}\left(\alpha_{ij}y_1^iy_2^j+\beta_{ij}y_1^ie_1y_1^j\right)
\end{equation} for some $\alpha_{ij},\beta_{ij}\in \mathbb{C}$.
On the other hand, $s_1f$ is already a linear combination of normal diagrams. If $fs_1=s_1f$, then using that $S^\bullet_{a,a}$ is a basis, and reading off the terms with the underlying undotted diagram $s_1$, we get $s_1f(y_2,y_1)=s_1f(y_1,y_2)$, and so $f(y_2,y_1)=f(y_1,y_2)$.
\item For $\hbar=0$ and $f$ symmetric in $y_1,y_2$, equation~\eqref{eqn:fsi} turns into the equalities
$f(y_1,y_2)s_1=s_1f(y_2,y_1)=s_1f(y_1,y_2),$
thus $fs_1=s_1f$. \qedhere
\end{enumerate}
\end{proof}

\subsection{Some central elements}
Consider the following elements in  $\mathbb{C}[\hbar][y_1,\ldots , y_a]$:
$$z_{ij}=(y_i-y_j)^2, \text{ for }1\leq i\not=j\leq a\quad \text{and}\quad
D_{\hbar}=\prod_{1\le i<j\le a}(z_{ij}-\hbar^2).$$ Notice that the deformed squared Vandermonde determinant $D_{\hbar}$ is symmetric, $D_{\hbar}\in \mathbb{C}[\hbar][y_1,\ldots, y_a]^{S_a}$.
We will use these to produce central elements in $A_{\hbar}$. 

\begin{lemma}For any $1\le i\le a-1$, we have in $A_{\hbar}$ the equality
 $$e_i\cdot (z_{i,i+1}-\hbar^2)=(z_{i,i+1}-\hbar^2)\cdot e_i=0,$$ and consequently
 $D_{\hbar}e_i=e_i D_{\hbar}=0$.\label{centre-ei}
\end{lemma}
\begin{proof}
Using \ref{sym} (i), we get
$
e_i\cdot (z_{i,i+1}-\hbar^2)=e_i (y_{i+1}-y_i)^2 -\hbar^2e_i=\hbar e_i(y_{i+1}-y_i) -\hbar^2e_i=\hbar^2 e_i -\hbar^2 e_i=0,
$
which implies $e_iD_{\hbar}=0$.
The claim $D_{\hbar}e_i=0$ is proved analogously. 
\end{proof}

\begin{lemma}For any $1\le k\le a-1$ we have $D_{\hbar}s_k=s_k D_{\hbar}$.\label{Dsi}
\end{lemma}
\begin{proof}We analyze the commutation of $s_k$ with different factors $(z_{ij}-\hbar^2)$ of $D_{\hbar}$ separately. 

Assume $i,j \notin \{k,k+1\}$. Then \ref{affbraid}(iii) says that $y_i$ and $y_j$ commute with $s_k$. Therefore, 
\begin{align}(z_{ij}-\hbar^2)s_k=s_k(z_{ij}-\hbar^2).\label{commute1}\end{align}

Now assume $i=k, j=k+1$. We claim that
\begin{align}(z_{k,k+1}-\hbar^2)s_k=s_k(z_{k,k+1}-\hbar^2).\label{commute2}\end{align}
To prove it, use \ref{skein} to calculate $(y_k-y_{k+1})s_k=s_k(y_{k+1}-y_k)-2\hbar$, and then 
\begin{align*}
(y_k-y_{k+1})^2s_k&=(y_k-y_{k+1})s_k(y_{k+1}-y_{k})-2\hbar(y_k-y_{k+1})\\
&=(s_k(y_{k+1}-y_k)-2\hbar)(y_{k+1}-y_k)-2\hbar(y_k-y_{k+1})\;=\;s_k(y_k-y_{k+1})^2.
\end{align*}

The remaining factors of $D_{\hbar}$ contain $z_{ij}$ with exactly one of $i,j$ in $\{k,k+1\}$. Since $z_{ij}=z_{ji}$, it suffices to consider $j\ne k,k+1$, and further assume $j>k+1$. We claim that
\begin{align}(z_{k,k+1}-\hbar^2)\left( (z_{k,j}-\hbar^2)(z_{k+1,j}-\hbar^2)s_k\right)=(z_{k,k+1}-\hbar^2)\left(s_k(z_{k,j}-\hbar^2)(z_{k+1,j}-\hbar^2)\right).\label{commute3}\end{align}
To prove \eqref{commute3}, let us first calculate
\begin{align}
z_{k,j}s_k&=(y_k-y_j)^2s_k\;=\;(y_k-y_j)s_k(y_{k+1}-y_j)+\hbar(y_k-y_j)(e_k-1)\notag\\
&=s_kz_{k+1,j}+\hbar(e_k-1)(y_{k+1}-y_j)+\hbar(y_k-y_j)(e_k-1).\notag
\end{align}
From this and Lemma~\ref{centre-ei}, we get \
\begin{align}
(z_{k,k+1}-\hbar^2)(z_{k,j} -\hbar^2)s_k&=(z_{k,k+1}-\hbar^2)\left( s_k (z_{k+1,j} -\hbar^2) -\hbar (y_k+ y_{k+1}-2y_j)\right).
\label{zijskstep1}
\end{align}
Similarly, 
\begin{align}
(z_{k,k+1}-\hbar^2)(z_{k+1,j} -\hbar^2)s_k&=(z_{k,k+1}-\hbar^2)\left( s_k (z_{k,j} -\hbar^2) +\hbar (y_k+ y_{k+1}-2y_j)\right).
\label{zijskstep2}
\end{align}
Using \eqref{zijskstep1} and \eqref{zijskstep2}, we then obtain \eqref{commute3}, since $(z_{k,k+1}-\hbar^2)\left( (z_{k,j}-\hbar^2)(z_{k+1,j}-\hbar^2)s_k\right)$ equals
$$(z_{k,k+1}-\hbar^2)\left[s_k(z_{k,j}-\hbar^2)(z_{k+1,j}-\hbar^2) +(\hbar (y_k+ y_{k+1}-2y_j)-\hbar (y_k+ y_{k+1}-2y_j))(z_{k,j}-\hbar^2)\right]$$ which is however the same as
$(z_{k,k+1}-\hbar^2)\left(s_k(z_{k,j}-\hbar^2)(z_{k+1,j}-\hbar^2)\right)$. Thus \eqref{commute3} holds. Finally, \eqref{commute1}, \eqref{commute2} and \eqref{commute3} imply $D_{\hbar}s_k=s_k D_{\hbar}$. 
\end{proof}

\begin{lemma} Let $1\le i \le a-1$, and let $\tilde{f}\in \mathbb{C}[\hbar][y_1,\ldots , y_a]$ be  symmetric in $y_i,y_{i+1}$. Then there exist polynomials $ p_j= p_j(y_1,\ldots , y_a)\in \mathbb{C}[\hbar][y_1,\ldots , y_a]$ such that 
$$\tilde{f}s_i=s_i\tilde{f}+\sum_{j=0}^{\deg{\tilde{f}}-1}  y_{i}^j \cdot e_i \cdot  p_j.$$
\label{centre-tildefsi}

\end{lemma}
\begin{proof}
Analogues of the formulas in Lemma~\ref{lem:dot-on-cup} and~\ref{AMR2.3} imply that for any $k$,
\begin{align*}
(y_i^k+y_{i+1}^k)s_i&=s_i(y_i^k+y_{i+1}^k)+\hbar \sum_{j=0}^{k-1} \left(y_i^{k-1-j}e_i y_{i+1}^j+ y_{i+1}^je_i y_{i}^{k-1-j}\right)\\
&=s_i(y_i^k+y_{i+1}^k)+\hbar \sum_{j=0}^{k-1} y_i^{k-1-j}e_i y_{i+1}^j+ \sum_{j=0}^{k-1}\sum_{\ell=0}^{j} \hbar^{1+j-\ell}(-1)^{j+\ell} y_{i}^\ell e_i y_{i}^{k-1-j}.
\end{align*}
Thus, the claim holds for $\tilde{f}=y_i^k+y_{i+1}^k$. It also trivially holds for $\tilde{f}=y_j$ if $j\ne i,i+1$, as such $y_j$ commute with $s_i$. Finally, note that if the claim holds for $\tilde{f}_1$ and $\tilde{f}_2$, it also holds for $\tilde{f}_1\tilde{f}_2$ and $\tilde{f}_1+\tilde{f}_2$. On the other hand, the algebra of polynomials symmetric in $y_i,y_{i+1}$ is generated by the $y_i^k+y_{i+1}^k$, $k\ge 1$, and $y_j$'s with $ j\ne i,i+1$, and the claim follows.
\end{proof}

\begin{lemma}\label{lem:manycentralelts}
Let $\tilde{f}\in \mathbb{C}[\hbar][y_1,\ldots , y_a]^{S_a}$ be an arbitrary symmetric polynomial, and $c$ a constant. Then $f=D_{\hbar}\tilde{f}+c$ lies in the centre of $A_{\hbar}$. 
\end{lemma}
\begin{proof}The element $f$ is in $\mathbb{C}[\hbar][y_1,\ldots, y_a]$ so it commutes with $y_i$ for all $i$. By Lemma~\ref{centre-ei}, $$fe_i=\tilde{f}D_{\hbar}e_i+ce_i=ce_i=e_i D_{\hbar}\tilde{f}+ce_i=e_if.$$ 
Using Lemma~\ref{centre-tildefsi}, and then Lemmas  \ref{Dsi} and \ref{centre-ei}  we get 
\begin{align*}
fs_i& =D_{\hbar}\tilde{f}s_i+cs_i=D_{\hbar}\left( s_i\tilde{f}+\sum_j  y_{i}^j \cdot e_j \cdot  p_j \right)+cs_i= s_iD_{\hbar}\tilde{f}+s_ic=  s_i f.  \qedhere
\end{align*}
\end{proof}

\subsection{The centre of $\sVW_a$ and of $A_{\hbar}$}
\begin{proposition}\label{prop:Z(A0)}
	The centre $Z(A_0)$ of the graded VW superalgebra $\GVW_a$ consists of  all $f\in \C[y_1,\ldots, y_a]$ of the form 
	$f=D_{0}\tilde{f}+c$, for $\tilde{f}\in \C[y_1,\ldots, y_a]^{S_a}$ and $c\in \mathbb{C}$.  
\end{proposition}
\begin{proof}
	We showed in Lemmas \ref{centre-poly} and \ref{centre-sym} that $Z(A_0) \subseteq \C[y_1,\ldots, y_a]^{S_a}$, and that any symmetric polynomial commutes with $s_i$ for $1 \leq i \leq a-1$ and $y_j$ for $1 \leq j \leq a$. It remains to check which symmetric polynomials commute with $e_i$ for all  $1 \leq i \leq a-1$. To this end, fix $f\in Z(A_0)$. We will compute a condition on commutation with $e_1$; then the symmetry of $f$ will complete the proof.  
	
	Expanding $fe_1$ in the normal dotted diagram basis, the terms appearing with nonzero coefficient all have underlying (undotted) diagrams equal to $e_1$; i.e.\ $fe_1$ is a linear combination of terms of the form $y_1^ke_1p_{k}$  with $p_k \in \C[y_3,\ldots , y_a]$. Similarly, $e_1f$ is a linear combination of terms of the form $e_1y_1^k p_{k}$. Comparing, we get $p_k=0$ for $k>0$, and that $fe_1=p_{0}(y_3,\ldots , y_a)e_1$. Using the presentation of $A_0$ given in Definition \ref{def:Ahbar}, we have that a polynomial in the $y_i$'s is annihilated by $e_1$ if any only if it is a multiple of $(y_1-y_2)$ (see \ref{sym}, specializing to $\hbar = 0$). Thus 
	$$f = (y_1-y_2) g+ p_0, \quad \text{ with } g \in \mC[y_1,\ldots, y_a] \text{ and } p_0\in \mC[y_3,\ldots, y_a].$$

	We claim that $p_0 \in \C$, which will follow from the symmetry of $f$. For this let $b y_3^{\lambda_3}\cdots y_a^{\lambda_a}$ be a non-zero summand of  $p_0$, and write $y^\lambda =  y_4^{\lambda_4} \cdots y_a^{\lambda_a}$ for short. 
	If $\lambda_3\geq 1$, 
	then symmetry implies 
	$b y_1^{\lambda_3} y^{\lambda}$ is a term in $f$, so that 
	$by_1^{\lambda_3-1} y^{\lambda}$ is a term in $g$. So $-by_1^{\lambda_3-1} y_2 y^{\lambda}$, and therefore $-b y_2 y_3^{\lambda_3-1} y^{\lambda}$, are summands in $f$. Going back to $g$, we get that $b y_3^{\lambda_3-1} y^{\lambda}$ is a summand there, so that $b y_1 y_3^{\lambda_3-1} y^{\lambda}$ is a summand of $f$. But comparing the coefficient to that of $y_2 y_3^{\lambda_3-1} y^{\lambda}$, we see that this contradicts the symmetry of $f$. Therefore $\lambda_3 = 0$ for all non-zero summands of $p_0$, and thus by symmetry, $p_0 \in \CC$ as claimed. 
	
	Next, since $f$ is symmetric (specifically in $y_1$ and $y_2$), we have $g$ is antisymmetric in $y_1$ and $y_2$. Thus $g$ itself is a multiple of $(y_1 - y_2)$, i.e.\ $f-p_0$ is a multiple of $(y_1 - y_2)^2$. But now, since $f - p_0$ is symmetric, it must also be a multiple of $D_{0} = \prod_{1\le i<j\le a}(y_i-y_j)^2$. 
	So finally, $f$ is of the form 
	$$f=\prod_{1\le i<j\le a}(y_i-y_j)^2\cdot \tilde{f}+c=D_{0}\tilde{f}+c,$$
	for some symmetric polynomial $\tilde{f} \in  \C[y_1,\ldots, y_a]^{S_a}$ and constant $c \in \C$. 
\end{proof}

\subsection{The centre of $\sVW_a$}

The main result of this section, Theorem~\ref{centreA1}, describes the centre of  $\sVW_a$. To do that, we use the fact that the algebra $\sVW_a$ is a PBW deformation of the algebra $\GVW_a$, determine the centre of $\GVW_a$ and find a lift of the appropriate basis elements to $\sVW_a$. This approach differs from the common arguments for diagram algebras, where often the centre is realized as a subring of invariant polynomials satisfying certain cancellation properties, \cite{DRV2}. In our situation the cancellation properties did not appear very manageable, and we therefore omitted them. It would however be nice to know if an explicit result as Theorem~\ref{centreA1} could be achieved for instance for affine VW algebras as in \cite{Nazarov}, \cite{ES2}, BMW-algebras, see e.g. \cite{DRV2}, or walled Brauer algebras, see e.g. \cite{JK}, \cite{Sartori}. Compare also with \cite{C}, where the center of the Brauer superalgebra $\sBr_a$ is described in a similar way.

\begin{theorem}\label{centreA1}
The centre $Z(\sVW_a)$ of the VW superalgebra $\sVW_a=A_1$ consists of  all $f\in \C[y_1,\ldots, y_n]$, of the form 
$f=D_{1}\tilde{f}+c$, for $\tilde{f}\in \C[y_1,\ldots, y_a]^{S_a}$ an arbitrary symmetric polynomial and $c\in \mathbb{C}$.  
\end{theorem}

\begin{proof}
For any filtered algebra $B$ there exists a canonical injective algebra homomorphism $\varphi :\gr Z(B)\hookrightarrow Z(\gr(B))$, given for $f\in Z(B)^{\le k}$ by $\varphi(f+Z(B)^{\le (k-1)})= f+B^{\le (k-1)}$, see~\cite[6.13, 6.14]{MR}. For $B=\sVW_a$ and $\gr(B)=\GVW_a$, by Proposition~\ref{prop:Z(A0)} the centre of $A_0$ consists of elements of the form $f=D_{0}\tilde{f}+c$ for $\tilde{f}$ a symmetric polynomial and $c$ a constant. By Lemma~\ref{lem:manycentralelts}, $D_{1}\tilde{f}+c$ lies in the centre of $\sVW_a$, and we have $\varphi (c)=c$, and for $\tilde{f}$ symmetric and homogeneous of degree $k$, $\varphi(D_{1}\tilde{f}+\sVW_a^{\le a(a-1)+k-1})=D_{0}\tilde{f}$. Using Proposition~\ref{prop:Z(A0)}, we see that every $f\in Z(\GVW_a)$ is in the image of $\varphi$, so $\varphi$ is an isomorphism. 
\end{proof}

\begin{remark}
It is interesting to compare the description of the centre of $\sVW_a$ with \cite[Theorem 4.8]{Verapn}. It is shown there that the centre of $\mathcal{U}(\mathfrak{p}(n))/I$, where $I$ is the Jacobson radical of $\mathcal{U}(\mathfrak{p}(n))$, is isomorphic to the subring in the polynomial ring $\mathbb C[z_1,\dots z_n]$ of the form $\mathbb C\oplus\mathbb C[z_1,\dots,z_n]^{S_n}\Theta$, where $\Theta=\prod_{i<j}(z_i-z_j)^2$. In other words, this centre is isomorphic to $Z(\sVW_a)$ when $a=n$.
\end{remark}

\begin{theorem}
The centre $Z(A_{\hbar})$ of the superalgebra $A_{\hbar}$ consists of polynomials $f\in \C[\hbar][y_1,\ldots, y_n]$, of the form 
$f=D_{\hbar}\tilde{f}+c$, for $\tilde{f}\in \C[\hbar][y_1,\ldots, y_a]^{S_a}$ an arbitrary symmetric polynomial and $c\in \mathbb{C}[\hbar]$.  
\end{theorem}
\begin{proof}
The centre $Z(A_{\hbar})$ is by  Lemma~\ref{algiso} isomorphic to 
$Z(\op{Rees}(A_{1}))$, which is by  Lemma~\ref{ZRees} isomorphic to $\op{Rees}(Z(A_1))$. The centre $Z(A_1)$ consists by Theorem~\ref{centreA1} of elements of the form $f=D_{1}\tilde{f}+c$, with $\tilde{f}\in \C[y_1,\ldots, y_a]^{S_a}$ and $c\in \mathbb{C}$. Assume $\tilde{f}$ is homogeneous of degree $k$. Then $D_{1}\tilde{f}\in A_1^{\le k+a(a-1)}$, which gives an element $D_{1}\tilde{f}\hbar^{k+a(a-1)}$ of $\op{Rees}(Z(A_1))\cong Z(\op{Rees}(A_1))$. Using Lemma~\ref{algiso}, we see that $Z(A_{\hbar})$ is spanned by constants and the preimages under the isomorphism $A_{\hbar}\cong \op{Rees}(A_1)$ of elements $D_{1}\tilde{f}\hbar^{k+a(a-1)}$, 
which are equal to $D_{\hbar}\tilde{f}$. 
\end{proof}

\addtocontents{toc}{\protect\setcounter{tocdepth}{0}}

\end{document}